%% file: amsa01.tex
\documentclass{article}
\usepackage{arxiv}
\usepackage{setspace,amsmath,amssymb,amsthm,color}

\usepackage[backend=biber, 
    language=english,%
    style=authoryear,
	  citestyle=authoryear,
    uniquename=init,
    giveninits,
    dashed=false, 
    sorting=nyt, 
    maxbibnames=5, 
    maxcitenames=1, 
    bibencoding=utf8,
    doi=false,
    isbn=false,
    url=false,
    refsegment=section,
    defernumbers=true
    ]{biblatex}
\AtEveryBibitem{%
  \clearfield{pages}%
  \clearfield{month}%
  \clearlist{language}%
}

\newtheorem{Proposition}{Proposition}[section]
\newtheorem{Theorem}[Proposition]{Theorem}
\newtheorem{Lemma}[Proposition]{Lemma}
\newtheorem{Corollary}[Proposition]{Corollary}

\theoremstyle{definition}
\newtheorem{Definition}[Proposition]{Definition}
\newtheorem{Remark}[Proposition]{Remark}

\usepackage{cleveref}
\crefname{lstlisting}{listing}{listings}
\Crefname{lstlisting}{Listing}{Listings}

\crefname{equation}{equation}{equations}
\crefname{figure}{figure}{figures}
\crefname{Definition}{definition}{definitions}
\crefname{Proposition}{proposition}{propositions}
\crefformat{footnote}{#2\footnotemark[#1]#3}

\usepackage{amsfonts}
\usepackage{relsize}
\usepackage{paralist}
\usepackage{graphicx}
\usepackage{array, tabularx, framed}
\usepackage{longtable}
\usepackage[margin=40pt, labelfont=bf, format=plain, hypcap=false, labelsep= space]{caption}
\input{ChoiceOfVariables.tex}

\begin{document}
\title{{I}DENTIFYING {P}ROCESSES {G}OVERNING {D}AMAGE {E}VOLUTION IN {Q}UASI-{S}TATIC {E}LASTICITY\\{P}ART 1 - {A}NALYSIS}



\renewcommand{\shorttitle}{Damage Process Identification in Quasi-Static Elasticity}
\renewcommand{\headeright}{}
\author{ Simon Gr\"utzner \\
	Center for Industrial Mathematics\\
	University of Bremen\\
	Germany \\
	\texttt{simon.gruetzner@uni-bremen.de} \\
	\And
	Adrian Muntean\\
	Department of Mathematics and Computer Science\\
	University of Karlstad\\
	Sweden \\
	\texttt{adrian.muntean@kau.se}\\
}
\maketitle
\noindent
{\bf Abstract.}
\input{abstract.tex}
\section{Introduction}\label{sec:Introduction}
\input{introduction.tex}
\section{Setting and Weak Formulation}\label{sec:Setting}
\input{setting.tex}
\section{Forward Problem}\label{sec:Forward_Problem}
\input{forward_problem.tex}
\section{Inverse Problem}\label{sec:Inverse_Problem}
\input{inverse_problem.tex}
\section{Outlook}\label{sec:Outlook}
\input{outlook.tex}
\section*{Acknowledgements}
The authors are indebted to Michael B\"ohm (Bremen, Germany) for initiating and supporting this research.
\addcontentsline{toc}{section}{References}
\printbibliography
\end{document}

%% file: ChoiceOfVariables.tex
  \DeclareMathOperator*{\divergence}{div}
  
  \DeclareMathOperator*{\domain}{dom}
  \DeclareMathOperator*{\interior}{int}
  
  \newcommand{\constant}{\ensuremath{c}}

  \newcommand{\embedding}{\ensuremath{\mathcal{E}}}
  \newcommand{\elasticityoperator}{\ensuremath{A}}
  \newcommand{\timedependentelasticityoperator}{\ensuremath{\mathcal{A}}}
  \newcommand{\bilieckiconstant}{\ensuremath{\lambda}}
  \newcommand{\indicatorfunction}{\ensuremath{\chi}}
  \newcommand{\ball}{\ensuremath{\mathbb{B}}}
  \newcommand{\boundary}{\ensuremath{\Gamma}}
  \newcommand{\boundarydisplacements}{\ensuremath{\mbf{{\bar u}}}}
  \newcommand{\boundaryfunctions}{\ensuremath{W}}

  \newcommand{\continuousfunctions}{\ensuremath{C}}
  \newcommand{\convolution}{\ensuremath{\ast}}

  \newcommand{\damage}{\ensuremath{d}}
  \newcommand{\damageconstant}{\ensuremath{\omega}}
  
  \newcommand{\damagefunctions}{\ensuremath{\mathcal{D}}}
  \newcommand{\damageprocess}{\ensuremath{g}}
  \newcommand{\damageprocesses}{\ensuremath{\mathcal{G}}}

  \newcommand{\dimension}{\ensuremath{N}}
  
  \newcommand{\displacements}{\ensuremath{\mbf{u}}}
  \newcommand{\displacementdomain}[1]{%
  \ifnum\pdfstrcmp{#1}{}=0 %
    \ensuremath{\mathcal{V}} %
  \else %
    \ensuremath{\mathcal{V}_{#1}} %
  \fi %
  }

  \DeclareMathOperator*{\esssup}{ess\,sup}
  \newcommand{\integrablefunctions}{\ensuremath{L}}
  
  \newcommand{\forces}{\ensuremath{\mbf{f}}}

  \newcommand{\identity}{\ensuremath{\mathcal{I}}}
  
  \newcommand{\linearoperators}{\ensuremath{\mathcal{L}}}

  \newcommand{\measurablefunctions}{\ensuremath{\mathcal{M}}}
    \newcommand{\measure}[1]{%
    \ifnum\pdfstrcmp{#1}{H}=0 %
      \ensuremath{\mu}%
    \else %
      \ensuremath{\lambda} %
    \fi %
  }
  \newcommand{\mollifiedgradient}{\ensuremath{\nabla^\mu}}
  \newcommand{\mollifier}{\ensuremath{\psi_{\mu}}}
  \newcommand{\noiselevel}{\ensuremath{\delta}}
  
  \newcommand{\normals}{\ensuremath{\mbf{\nu}}}
  \newcommand{\perturbation}{\ensuremath{h}}
  \newcommand{\projection}{\ensuremath{\pi}}
  \newcommand{\ptsmap}{\ensuremath{F}}
  \newcommand{\qisobolevfunctions}{\ensuremath{H}}

  \newcommand{\forwardoperator}{\ensuremath{\Phi}}
  
  \newcommand{\radius}{\ensuremath{r}}

  \newcommand{\sobolevfunctions}{%
  \ensuremath{W}}
  \newcommand{\nemytskiioperator}{\ensuremath{G}}

  \newcommand{\spatialhilbertfunctions}{\ensuremath{H}}
  \newcommand{\spatialfunctions}{\ensuremath{V}}
  \newcommand{\spatialdomain}{\ensuremath{\Omega}}
  \newcommand{\spatialtestfunction}{\ensuremath{v}}
  \newcommand{\spatialtestfunctions}{\ensuremath{\mbf{v}}}
  \newcommand{\spatialvariable}{\ensuremath{x}}
  \newcommand{\spatialvariables}{\ensuremath{\mbf{x}}}

  \newcommand{\strains}{\ensuremath{\mbf{\varepsilon}}}
  
  \newcommand{\stresses}{\ensuremath{\mbf{\sigma}}}
  
  \newcommand{\stressforces}{\ensuremath{\mbf{\tau}}}
  
  \newcommand{\sym}{\ensuremath{\textnormal{Sym}}}
  \newcommand{\timevariable}{\ensuremath{t}}
  \newcommand{\timedomain}{\ensuremath{S}}
  \newcommand{\timedomainmax}{\ensuremath{T}}

  \newcommand{\traceoperator}{\ensuremath{\gamma}}

  \newcommand{\elasticitytensor}{\ensuremath{\mbf{\mathbb{E}}\ \!}}
  \newcommand{\simplefunctions}{\ensuremath{\mathcal{S}}}
  \newcommand{\fixedpointoperator}{\ensuremath{\Psi}}

  \newcommand{\transpose}{\ensuremath{\intercal}}

  \newcommand{\pd}[3]{%
    \ifnum\pdfstrcmp{#3}{1}=0 %
      \ensuremath{\frac{\partial #1}{\partial #2}}%
    \else %
      \ensuremath{\frac{\partial^#3 #1}{\partial #2^#3}} %
    \fi %
  }
  \newcommand{\mbb}[1]{\ensuremath{\mathbb{#1}}}
  \newcommand{\td}[3]{%
    \ifnum\pdfstrcmp{#3}{1}=0 %
      \ensuremath{\frac{d #1}{d #2}}%
    \else %
      \ensuremath{\frac{d^#3 #1}{d #2^#3}} %
    \fi %
  }
  \newcommand{\tn}[1]{\textnormal{#1}}
  
  \newcommand{\mbf}[1]{\ensuremath{\boldsymbol{#1}}}
  
  \newcommand{\fa}{\forall\ \!}

\renewcommand{\phi}{\varphi}
\renewcommand{\epsilon}{\varepsilon}

%% file: abstract.tex
We present a quasi-static elasticity model that accounts for damage evolution based on the ideas of \cite{Ka58} and \cite{YuRa68}. 
We show well-posedness of the resulting strongly nonlinear system of differential equations. 
The specific feature is the connection of displacements to damage evolution via a Nemytskii- or superposition- operator.
From a material modelling perspective, the shape of this operator defines the afforementioned connection.
The novelty in this work is the presentation of an inverse problem to identify the shape of this Nemytskii-operator. 
We establish the Fréchet-derivative of the forward operator as well as the adjoint of the derivative and characterize both via systems of linear differential equations. 
We prove ill-posedness of the inverse problem and provide a sufficient condition for the classical nonlinear Landweber method to converge.

%% file: introduction.tex
\noindent The demand for nonlinear material models is vastly growing.
One example being the modeling of damage processes due to different causes.
In general, any real material is subjected to damage evolution.
This causes crucial mechanical properties such as stiffness, load-carrying capacity or creep rupture time to change over time.
With a growing demand for materials specifically tailored to certain settings understanding damage processes is vital to increase predictability of mechanical models and their lifecycles.\\[2ex]
An important factor to keep in mind is that damage itself cannot be observed directly.
It must be measured indirectly by the effect it has on the material properties (cf.~\cite[subsection 1.4.5]{DuKr96}).
This constitutes an inverse problem (cf. \cite{EnHaNeu96,AnKi21,AnRi03}) as we have to estimate the \emph{cause} of an \emph{observed effect}.
So one way to address this would be in the shape of an parameter identification setting in which one would measure the displacements and stresses in a standard tensile test to infer the change of Lamé parameters.
Identifiability of Lamé parameters, stability, and different reconstruction approaches have been studied quite a lot in the Literature, see e.g. \cite{AlBaObHa14, BaBeImMo14, BaUhl12, ThGe20, LeSchl17} and many many more.
In contrast to these approaches, we do want to incorporate the effect damage has on the material response and identify the process that links damage to displacements itself (see \cite{GrMu17}).
This is the true motivation for our project.
We want to present a setting that allows to identify the process that connects displacements to damage.
Incorporating damage evolution already results in a strongly nonlinear system of differential equations, which makes establishing the ingredients that are needed for a successful investigation of the inverse problem much harder to come by.
Since we are looking to identify a process in the shape of a Nemytskii operator here, which to the best of the author's knowledge is a novelty in itself, additional difficulties are added to the task. \\[2ex]
We briefly want to give some references for introductory texts.
An introduction to general elasticity theory can be found in \cite{PhCi88,DuLi76,EbZe90}.
For the modelling of damage already exists a vast collection of research articles (see \cite{GrMu17} for an extensive review).
An introduction to damage modeling in a general continuum mechanical setting can be found in \cite{LeCh90, SuMu12}, for example.
As this article aims to lay the foundation for future research, we based the damage model on the works of \cite{Ka58} and \cite{YuRa68}, which still are applied today, see e.g. \cite{BoCaFreGraNa20, NaNiPoSba03}.\\[2ex]
To motivate our damage model and general setting, let us think of a damaged cylinder fixed at its bottom surface.
We apply a time dependent force $\stressforces$ on its top surface and measure the displacements, as it is common in tensile or creep rupture tests.
This kind of experiment represents the concept according to which we define our mathematical model (see \cite{GrMu17}).
\begin{Definition}[Quasi-static problem of linear elasticity in damaged continua]\label{def:problem_of_linear_elasticity}
    In the following, we refer to
    \begin{subequations}
      \begin{alignat}{2}\label{momentum_equation}
        \stresses %
        &= (1-\damage)\elasticitytensor\strains(\displacements), &\quad \textnormal{in }&\timedomain\times\spatialdomain,\\
        -\divergence{(\stresses)} %
	&= \forces,  &\quad  \tn{in }&\timedomain\times\spatialdomain, \\
        \damage' &= (1-\damage)^{-\alpha}\damageprocess(\nabla\displacements), &\qquad \textnormal{in }&\timedomain\times\spatialdomain \\
        \displacements 	&= \mbf{0},  &\qquad \tn{on }&\timedomain\times\boundary_0,	 \\
        \stresses\normals	&= \stressforces,  &\qquad\tn{on }&\timedomain\times\boundary_1,
      \end{alignat}
      with initial damage
      \begin{equation}
        \damage(0)=\damage_0\quad\tn{in }\spatialdomain
      \end{equation}
    \end{subequations}
    as the \emph{quasi-static problem of linear elasticity in damaged continua}.
\end{Definition}
We state the general outline for this work.
\Cref{sec:Setting} introduces the general notations, preliminary results and general setting.
The analysis of the forward operator (cf.~\Cref{def:forward_operator}) is treated in Section \ref{sec:Forward_Problem}.
We show that it is actually well-defined and provide conditions to ensure its differentiability.
We close this section by presenting a characterisation of the Hilbert adjoint of the Fr\'echet-derivative.
In \Cref{sec:Inverse_Problem}, we state the resulting inverse problem and analyze ill-posedness in the linearized and fully nonlinear setting.
We then look at a strong nonlinearity condition that is a vital property to ensure convergence of iterative Hilbert space methods like nonlinear \emph{Landweber} or \emph{REGINN} (see \cite{HaNeuSche95, AnRi99, AnRi01, }).
We close with an outlook to outline possible future research.

%% file: setting.tex
  We denote by $\timedomain:=(0,\timedomainmax)$ the time interval of interest, with $0<\timedomainmax<\infty$.
  For a fixed $\dimension\in\{1,2,3\}$, let $\spatialdomain\subset\mathbb{R}^{\dimension}$ be a bounded Lipschitz domain, whose boundary $\partial\spatialdomain$ decomposes into mutually disjoint sets $\boundary_0$ and $\boundary_1$ where $\boundary_0, \boundary_1$  are closed with positive surface measures.
  Both are assumed to be unions of \emph{connected components} of $\partial\spatialdomain$ to avoid issues that arise from singularities where the type of boundary condition change (see also Remark \ref{rem:comp_cond} item \eqref{rem:comp_cond:add_options}).%
  We introduce basic function spaces for damage evolution and start with
  \begin{equation}\label{eq:damage_functions}
      \damagefunctions:=\left\{%
     \damage\in\sobolevfunctions^{1,\infty}\!\left(\timedomain;\, \integrablefunctions^\infty(\spatialdomain)\right);\ %
       0\le\damage(\timevariable,\spatialvariables)\le\damageconstant_1 %
      \text{ a.e. in } \timedomain\times\spatialdomain \right\}
  \end{equation}
  where $\damageconstant_1\in\mbb{R}$ denotes a fixed non-negative constant such that $0\le\damageconstant_1 <1$ holds and name its elements \emph{damage functions}.
  Note that this is a closed set in $\sobolevfunctions^{1,\infty}(\timedomain;\integrablefunctions^{\infty}(\spatialdomain))$ with respect to its usual norm.
  $\sobolevfunctions^{k,p}(\spatialdomain)$ denotes the specified Sobolev space of $k$-times weakly differentiable p-integrable functions.
  We denote by $\spatialhilbertfunctions^k(\spatialdomain):=\sobolevfunctions^{k,2}(\spatialdomain)$ and $\integrablefunctions^p(\spatialdomain):=\sobolevfunctions^{0,p}(\spatialdomain)$ square integrable Sobolev and Lebesgue spaces, respectively.
  The restrictions on the constant $\damageconstant_1$ model the fact that we only consider partial damage.
  We speak of partial damage opposed to substantial damage if the material bonds of a material do not fully disintegrate.
  From a technical point of view this is crucial to preserve the strong ellipticity of the elasticitytensor.
  We will refer to elements of %
  \begin{equation}\label{eq:initial_damage}
    \damagefunctions_0:= %
    \{\damage_0\in L^\infty(\spatialdomain);\, %
    0\le\damage_0(\spatialvariables)\le\damageconstant_0 %
    \text{ a.e. in }\spatialdomain\}
  \end{equation}
  as \emph{initial damage} and expect the constant $\damageconstant_0$ to suffice $0\le\damageconstant_0\le\damageconstant_1$.
  This is to allow for non-zero initial damage as well.
  Let $\overline{Y}:=\overline{\mathbb{B}}(0;\overline{y})\subset\mbb{R}^{\dimension^2}$ be the closed ball of radius $\bar y>0$ and center in $0$ (see also Remark \ref{rem:comp_cond} item \eqref{rem:comp_cond:bound_on_y}).
  We introduce the set of admissible \emph{damage processes}
  \begin{equation}\label{eq:damage_sources}
    \begin{aligned}
      \damageprocesses:=\Big\{\damageprocess %
      &\in  \integrablefunctions^\infty\Big(\timedomain; \integrablefunctions^\infty\big(\spatialdomain;\continuousfunctions^{1,1}(\overline{Y})\big)\Big); \\ %
      &\fa\mbf{y}\in\overline{Y}:\ %
      0\le\damageprocess(\cdot,\cdot,\mbf{y})%
      \le T^{-1}(\damageconstant_1 %
      - \damageconstant_0)(1-\damageconstant_1)^\alpha %
      \text{ a.e. in }\timedomain\times\spatialdomain\
      \Big\}
    \end{aligned}
  \end{equation}
  where $\alpha\ge 1$ is some fixed constant and $\continuousfunctions^{m,\lambda}(\overline{Y})$ names the specified H\"older space of $m$-times continuously differentiable functions $\overline{Y}\subset\mathbb{R}^{\dimension^2}\to\mathbb{R}$ satisfying the H\"older condition of exponent $\lambda$.
  Note that $\continuousfunctions^{0,1}(\overline{Y})$ instead of $\continuousfunctions^{1,1}(\overline{Y})$ in \eqref{eq:damage_sources} is sufficient to show well-posedness of the forward problem.
  But more regularity is needed in order to prove differentiability, see Lemma \ref{lem:nem-op_well-posed}.
  We point out that the mapping $(\timevariable,\spatialvariables)\mapsto\damageprocess(\timevariable,\spatialvariables,\mbf{y})$ is an element of $\integrablefunctions^\infty(\timedomain;\,\integrablefunctions^\infty(\spatialdomain))$ for all $\mbf{y}\in\overline{Y}$, that $\damageprocesses$ is closed in $\integrablefunctions^\infty(\timedomain;\,\integrablefunctions^\infty(\spatialdomain;\,\continuousfunctions^{1,1}(\overline{Y})))$, and that $\mbf{y}\mapsto\damageprocess(\cdot, \cdot,\mbf{y})$ is Lipschitz continuous a.e.~in $\timedomain$ and $\spatialdomain$.
  \begin{Remark}\label{rem:comp_cond}
    \begin{inparaenum}[(a)]
      \item \label{rem:comp_cond:add_options}%
      In a more general setting additional compatibility conditions are needed to ensure higher regularity of the solution (see \cite{MiMi07, GiSa97}, and the references therein).
      Also compare, e.g., \cite{Sha68} where the author shows that for an elliptic problem with homogeneous boundary conditions and right-hand side $f\in\integrablefunctions^p(\spatialdomain)$ guarantees $u\in\sobolevfunctions^{s,p}(\spatialdomain)$ for all $s<\frac{1}{2}+\frac{2}{p}$.
      Particularly, this implies $u\in\continuousfunctions^{\frac{1}{2}-\varepsilon}$ for $\varepsilon>0$.
      An approximation via suitable Robin type boundary conditions is a viable approach to overcome the drawback of mixed boundary conditions and the singularities that come with it (see \cite{LiMa72} for regularity and \cite{GiAu18} for approximation properties).
      Unfortunately, we will see in Section \ref{sec:strong_tangential_cone_condition} that our chosen approach to prove the nonlinear tangential cone condition excludes Robin type boundary conditions. \\[2ex]
      \item \label{rem:comp_cond:bound_on_y}%
      Note that boundedness and closedness for $\overline{Y}\subset\mathbb{R}^{\dimension^2}$ is needed so $\continuousfunctions^{m,\lambda}(\overline{Y})$ becomes a Banach space, which is a necessity to investigate inverse problems with iterative Hilbert or Banach space methods.
      To introduce well-defined and differentiable Nemytskii operators $\continuousfunctions^{m,\lambda}_{\text{loc}}(\mathbb{R}^{\dimension^2})$ would suffice.
    \end{inparaenum}
  \end{Remark}
  By taking the principle ideas from \cite{ApZa90, FrTr10} and extending those by necessary adjustments to fit to our specific setting, these properties ensure a well-defined superposition or Nemytskii operator. %
\begin{Lemma}\label{lem:nem-op_well-posed}
  For every $\boldsymbol{f}\in\integrablefunctions^\infty(\timedomain;\integrablefunctions^\infty(\spatialdomain))^{\dimension^2}$ an admissible damage process $\damageprocess\in\damageprocesses$ generates a Lipschitz-continuous Nemytskii operator
  \begin{subequations}\label{eq:def_nem-op}
  \begin{equation}\label{eq:def_nem-op_a}
    \nemytskiioperator\colon\integrablefunctions^\infty(\timedomain;\integrablefunctions^\infty(\spatialdomain))^{\dimension^2}\to\integrablefunctions^\infty(\timedomain;\integrablefunctions^\infty(\spatialdomain))
  \end{equation}
  via
  \begin{equation}\label{eq:def_nem-op_b}
    \nemytskiioperator(\boldsymbol{f})(\timevariable,\spatialvariables):=\damageprocess(\timevariable,\spatialvariables,\boldsymbol{f}(\timevariable,\spatialvariables)).
  \end{equation}
  If also $\|\boldsymbol{f}\|\le \bar y$ holds
  \begin{equation}\label{eq:nem-op_bounds}
    0\le\nemytskiioperator(\boldsymbol{f})\le\timedomainmax^{-1}(\damageconstant_1-\damageconstant_0)(1-\damageconstant_1)^\alpha
  \end{equation}
  \end{subequations}
  is satisfied almost everywhere in $\timedomain$ and $\spatialdomain$.
\end{Lemma}
\begin{proof}
Let $\boldsymbol{f}\in$ $\integrablefunctions^\infty\big(\timedomain;$ $\integrablefunctions^\infty(\spatialdomain)^{\dimension^2}\big)$ with $\|\boldsymbol{f}\|\le\bar y$ be fixed.
Boundedness of $\timevariable\mapsto\nemytskiioperator(\boldsymbol{f})(\timevariable,\cdot)\in\integrablefunctions^\infty\big(\timedomain;\integrablefunctions^\infty(\spatialdomain)\big)$ immediately follows from the definition of $\damageprocesses$.
It remains to show measurability.
To this end, we introduce functions
\begin{equation}\label{eq:nem-op_fixed_time}
  \damageprocess(\timevariable,\cdot)\in\integrablefunctions^\infty\big(\spatialdomain;\continuousfunctions^{1,1}(\overline{Y})\big),\quad %
  \boldsymbol{f}(\timevariable,\cdot)\in\integrablefunctions^\infty(\spatialdomain)
\end{equation}
for almost every $\timevariable\in\timedomain$.
These functions are measurable regarding the spatial variable.
Hence, there exist sequences of simple functions $\damageprocess_j(\timevariable,\cdot)\in\simplefunctions\big(\spatialdomain;\continuousfunctions^{1,1}(\overline{Y})\big)$ and $\boldsymbol{f}_j(\timevariable,\cdot)\in\simplefunctions(\spatialdomain)$ such that
  \begin{equation}\label{eq:nem-op_converging_simple_funtions_fixed_time}
    \damageprocess_j(\timevariable,\cdot)\to\damageprocess(\timevariable,\cdot)\text{ in } \continuousfunctions^{1,1}(\overline{Y}), \quad %
    \boldsymbol{f}_j(\timevariable,\cdot)\to\boldsymbol{f}(\timevariable,\cdot)\text{ in }\mathbb{R}^{\dimension^2}
  \end{equation}
  for $j\to\infty$ pointwise almost everywhere in $\spatialdomain$.
  We now take the composed sequence $\nemytskiioperator_j(\boldsymbol{f}_j)(\timevariable,\cdot):=(\damageprocess_j(\timevariable,\cdot))(\boldsymbol{f}_j(\timevariable,\cdot))$ in $\simplefunctions(\spatialdomain)$ and fix $\spatialvariables\in\spatialdomain$.
  This allows for
    \begin{multline}\label{eq:nem-op_conv_in_ess_bounded_func}
        |\nemytskiioperator_j(\boldsymbol{f}_j)(\timevariable,\spatialvariables)-\nemytskiioperator(\boldsymbol{f})(\timevariable,\spatialvariables)| \\
        \le \left\|\damageprocess_j(\timevariable,\spatialvariables)\right\|_{\continuousfunctions^{0,1}(\overline{Y})}|\boldsymbol{f}_j(\timevariable,\spatialvariable)-\boldsymbol{f}(\timevariable,\spatialvariables)|_{\mathbb{R}^{\dimension^2}}
        +\|\damageprocess_j(\timevariable,\spatialvariables)-\damageprocess(\timevariable,\spatialvariables)\|_{\continuousfunctions^{0,1}(\overline{Y})} \to 0
    \end{multline}
  for $j\to\infty$ and, thus, proving $(\spatialvariables\mapsto\nemytskiioperator(\boldsymbol{f})(\timevariable,\spatialvariable))\in\integrablefunctions^\infty(\spatialdomain)$ for all $\boldsymbol{f}$ and almost every $\timevariable\in\timedomain$.
  To see $(\timevariable\mapsto\nemytskiioperator(\boldsymbol{f})(\timevariable,\cdot))\in\integrablefunctions^\infty(\timedomain;\integrablefunctions^\infty(\spatialdomain))$ we employ similar arguments to ensure existence of sequences of simple functions $\damageprocess_j\in\simplefunctions\big(\timedomain;\integrablefunctions^\infty(\spatialdomain;\continuousfunctions^{1,1}(\overline{Y}))\big)$, $\boldsymbol{f}_j\in\simplefunctions\big(\timedomain;\integrablefunctions^\infty(\spatialdomain)^{\dimension^2}\big)$ converging pointwise almost everywhere in $\timedomain$, i.e.,
\begin{equation}\label{eq:nem-op_conv_simple_functions_fixed_position}
  \damageprocess_j(\timevariable)\to\damageprocess(\timevariable)\text{ in }\integrablefunctions^\infty(\spatialdomain;\continuousfunctions^{1,1}(\overline{Y})),\quad %
  \boldsymbol{f}_j(\timevariable)\to\boldsymbol{f}(\timevariable)\text{ in }\integrablefunctions^\infty(\spatialdomain)^{\dimension^2}
\end{equation}
for $j\to\infty$. We argue as before in \eqref{eq:nem-op_conv_in_ess_bounded_func} and infer
\begin{multline}\label{eq:nem-op_conv_in_time_and_space}
  |\nemytskiioperator_j(\boldsymbol{f}_j)(\timevariable,\spatialvariables)-\nemytskiioperator(\boldsymbol{f})(\timevariable,\spatialvariables)| \\
          \le \|\damageprocess_j(\timevariable)\|_{\integrablefunctions^{\infty}(\spatialdomain;\continuousfunctions^{0,1}(\overline{Y}))}\|\boldsymbol{f}_j(\timevariable)-\boldsymbol{f}(\timevariable)\|_{\integrablefunctions^{\infty}(\spatialdomain)^{\dimension^2}}
          +\|\damageprocess_j(\timevariable)-\damageprocess(\timevariable)\|_{\integrablefunctions^\infty(\spatialdomain;\continuousfunctions^{1,1}(\overline{Y}))}
      \to 0,
\end{multline}
  for $j\to\infty$ and for almost every $\timevariable\in\timedomain$.
  Using this last equation to infer Lipschitz-continuity for $\nemytskiioperator$ as well as proving \eqref{eq:nem-op_bounds} by using the bounds defined in $\eqref{eq:damage_sources}$ is straight-forward.
\end{proof}
  For the momentum balance equation we introduce the basic spatial function spaces
  \begin{equation}\label{eq:basic_spatial_functions}
    \spatialfunctions:=\{\spatialtestfunctions\in\sobolevfunctions^{1,2}(\spatialdomain)^{\dimension};\  \spatialtestfunctions = 0\textnormal{ on }\boundary_0\}, \quad %
    \boundaryfunctions:=\sobolevfunctions^{\frac{1}{2},2}(\boundary_1)^{\dimension}
  \end{equation}
  and denote their respective duals by $\spatialfunctions^*$ and $\boundaryfunctions^*$.
  The space of \emph{admissible tractions $\stressforces$ on $\boundary_1$} is denoted by $\integrablefunctions^\infty(\timedomain;\boundaryfunctions^*)$ and endowed with its usual norm.
  Later we will also make use of higher regularity results and change to a $\sobolevfunctions^{k,p}$ setting when need be.
  We introduce a family of operators $\elasticityoperator_{\damage}(\timevariable):\spatialfunctions\to\spatialfunctions^*$ for almost every $\timevariable\in\timedomain$ and its realization $\mathcal{\elasticityoperator}(\damage):\integrablefunctions^2(\timedomain;\spatialfunctions)\to \integrablefunctions^2(\timedomain;\spatialfunctions^*)$ via
    \begin{equation}\label{eq:Operator_A} %
      \begin{aligned}
        \big\langle\left(\mathcal{\elasticityoperator}(\damage)\displacements\right)(\timevariable),\spatialtestfunctions\big\rangle_{\integrablefunctions^2(\spatialdomain)} %
          &:= \big\langle \elasticityoperator_{\damage}(\timevariable)(\displacements(\timevariable)), \spatialtestfunctions\big\rangle_{\integrablefunctions^2(\spatialdomain)} \\
          &:= \int_{\spatialdomain} (1-\damage(\timevariable))\elasticitytensor\strains(\displacements(\timevariable)):\strains(\spatialtestfunctions)\, d\boldsymbol{\spatialvariable}.
      \end{aligned}
    \end{equation}
  Here, $\strains\in\sym(\dimension)$ denotes the Cauchy-Green strain tensor, i.e., $\strains(\displacements):=\frac{1}{2}(\nabla\displacements + \nabla\displacements^{\transpose})$ and $\sym(\dimension)$ refers to the set of symmetric $\dimension^2$ matrices.
  The symbol $\elasticitytensor$ denotes the fourth-order elasticity tensor, whose components are thought of as functions of only spatial coordinates (see, for example, \cite{PhCi88,MaHu94,EbZe90} for details).
  This is a common assumption in an anistropic inhomogeneous linear elastic setting, i.e., $\elasticitytensor_{ijkl}\in\integrablefunctions^\infty(\spatialdomain)$ for all $i$, $j$, $k$, $l$ $\in\{1,2,\dots,\dimension\}$.
  Additionally, the elasticity tensor is symmetric (see e.g.~\cite[Section 61.4D]{EbZe90}), i.e., $\elasticitytensor_{ijkl}=\elasticitytensor_{jikl}=\elasticitytensor_{klji}$ for all $i$, $j$, $k$, $l$ $\in\{1,2,\dots,\dimension\}$, and uniformly elliptic (see e.g.~\cite[ Chapter 3]{DuLi76}), i.e., there exists a positive real constant $c$ such that $\elasticitytensor\strains:\strains\ge c\strains:\strains$ for all $\strains\in\sym(\dimension)$. \\[2ex]
  To present a well-posed forward problem we introduce a regularized version of the displacement gradient.
  This is owed to the fact that we need to ensure an estimate alike the one presented that will be presented in \eqref{eq:Lipschitz_estimate_for_nem-op}.
  Such an estimate is needed to prove that the resulting fixed-point operator is k-contractive (see Proof of Theorem \ref{theo:direct_problem_well-posed}).
  Also see first paragraph in \Cref{subsec:coupled_problem}.
  To this end, we introduce the \emph{mollified gradient} via
  \begin{equation}\label{def:mollified_gradient}
      \mollifiedgradient\displacements_i := D_i^{\mu}\displacements := \mu^{-1}(\displacements(\cdot + \mu\boldsymbol{e}_i) - \displacements).
  \end{equation}
  We also employ the trace theorem to identify $\stressforces(\timevariable)\in\boundaryfunctions^*$ with $\stressforces(\timevariable)\circ\traceoperator\in\spatialfunctions^*$, where $\traceoperator$ denotes the trace operator.
  Henceforth, we will only write $\stressforces(\timevariable)\in\spatialfunctions^*$.
  \begin{Remark}\label{rem:on_mollifiers}
    We give some examples for different types of mollifiers that work in this setting.\\[2ex]
    \begin{inparaenum}[(a)]
      \item \label{rem:on_mollifiers:smooth_functions}%
      Let $\mu\in\mathbb{R}$ satisfy $\mu>0$. Using \emph{convolution with smooth functions} alike
      \begin{equation}
          \mollifier(\spatialvariables):=\left\{%
          \begin{array}{rl}
              \mu_0\exp\left(\frac{1}{|\spatialvariables|^2-\mu^2}\right),%
                & \tn{ for }|\spatialvariables|<\mu \\
               0\ , &\tn{ for }|\spatialvariables|\ge\mu,
          \end{array}
          \right.
      \end{equation}
      with $\mu_0$ such that $\int_{\mbb{R}^{\dimension}}\mollifier(\xi)\,d\xi=1$ or \emph{local spatial averaging via convolution with an indicator function} like
      $\mollifier:=\indicatorfunction_{\ball(0;\mu)}$ where $\bar{\mu}=\int_{\mbb{R}^{\dimension}}\indicatorfunction_{\ball(0;\mu)}(\xi)\,d\xi$.
      Then the \emph{mollified gradient of $f\in\sobolevfunctions_{loc}^{1,1}\left(\spatialdomain\right)$} can be introduced as $(\mollifiedgradient f)_i:=\mollifier\convolution\partial_i f$ for $i=1,\dots,\dimension$, where the right-hand side denotes the convolution product of $\mollifier$ and $\partial_i f$ for $i=1\dots\dimension$. \\[2ex]
      \item\label{rem:on_mollifiers:projection}%
      Approximation of $\nabla\displacements$ by finite dimensional subspaces through the orthogonal $L^2$-projection $\Pi:\integrablefunctions^2(\spatialdomain)\to\sobolevfunctions_h^{k,2}(\spatialdomain)$.
      Here, $\sobolevfunctions_h^{k,2}(\spatialdomain)$ denotes a finite dimensional subspace of $\sobolevfunctions^{k,2}(\spatialdomain)$.
      These kind of projections are common practive in Finite Element applications.
    \end{inparaenum}
  \end{Remark}
  After this preliminary work we are able to state the traction-driven problem in weak formulation.%
  \begin{Definition}[The traction-driven problem]\label{def:direct_problem}
    Provided that
      \begin{equation}\label{eq:data_pts_map}
        \forces\in \integrablefunctions^\infty(\timedomain;\spatialfunctions^*), \quad \stressforces\in \integrablefunctions^\infty(\timedomain;\boundaryfunctions^*),\quad\damage_0\in\damagefunctions_0, \quad\damageprocess\in\damageprocesses
      \end{equation}
      holds, we search for $\displacements\in \integrablefunctions^\infty(\timedomain;\spatialfunctions)$ fulfilling
    \begin{subequations}\label{eq:direct_problem}
      \begin{alignat}{2} \label{eq:EqOfMo}
        \mathcal{\elasticityoperator}(\damage)\displacements &= \forces + \stressforces & \quad\textnormal{in } & \integrablefunctions^\infty(\timedomain;\spatialfunctions^*),\\ \label{eq:DaEvEq}
        \damage' &= \big(1-\damage\big)^{-\alpha}\nemytskiioperator(\mollifiedgradient\displacements) &\quad\textnormal{in }& \integrablefunctions^\infty\big(\timedomain;\integrablefunctions^\infty(\spatialdomain)\big),\\ \label{eq:DaEvEq_initial_condition}
        \damage(0) &= \damage_0 &\quad\textnormal{in }& \integrablefunctions^\infty(\spatialdomain).
      \end{alignat}
    \end{subequations}%
  \end{Definition}
\begin{Remark}
  \begin{inparaenum}[(a)]
  \item
  We comment on the fact that \eqref{eq:EqOfMo} is well-defined.
  Since time is merely a parameter here,  an initial value for the displacements results from given $\forces$, $\stressforces$, and respective boundary conditions.
  In case of higher regularity in time we have $\displacements(0)=\mathcal{\elasticityoperator}(\damage)^{-1}(\forces(0) + \stressforces(0))$.\\[2ex]
  \item
  Under the given assumptions, we have $z\mapsto z^{-\alpha}\in\continuousfunctions([1-\damageconstant_1,1-\damageconstant_0])$ and thus $(1-d)^{-\alpha}\in\integrablefunctions^\infty(\timedomain;\integrablefunctions^\infty(\spatialdomain))$.
  Since $\nemytskiioperator(\mollifiedgradient\displacements)\in\integrablefunctions^\infty(\timedomain;\integrablefunctions^\infty(\spatialdomain))$ holds, \eqref{eq:DaEvEq} is presented in a meaningful way due to the fact that products of essentially bounded functions are essentially bounded, too.
  \end{inparaenum}
\end{Remark}
As will be proven in Theorem \ref{theo:direct_problem_well-posed}, this introduces an operator $\ptsmap\colon \integrablefunctions^\infty(\timedomain;\spatialfunctions^*)\times \integrablefunctions^\infty(\timedomain;\boundaryfunctions^*)\times\damagefunctions_0\times\damageprocesses\to\integrablefunctions^\infty(\timedomain;\spatialfunctions)\times\damagefunctions$ mapping forces $\forces$, traction $\stressforces$, initial damage $\damage_0$ and a damage process parameter $\damageprocess\in\damageprocesses$ to unique displacements $\displacements$ and damage $\damage$ thus solving the traction-driven problem.
Henceforth, we will fix $\forces$, $\stressforces$, $\damage_0$ and let $\ptsmap\colon\damageprocesses\to\integrablefunctions^\infty(\timedomain;\spatialfunctions)\times\damagefunctions$ denote the parameter-to-state map.
\begin{Definition}[Forward Operator]\label{def:forward_operator}
  We introduce the \emph{forward operator} $\forwardoperator\colon \damageprocesses\to \integrablefunctions^2(\timedomain\times\spatialdomain)^{\dimension}$ via $\Phi:=\pi_1\circ\ptsmap$, where $\pi_1$ denotes the projection onto the first component of $\ptsmap$ and we interpret $\integrablefunctions^{\infty}(\timedomain;\spatialfunctions)$ as a subset of $\integrablefunctions^2(\timedomain\times\spatialdomain)^{\dimension}$.
\end{Definition}

%% file: forward_problem.tex
In this section we will focus on the forward problem and start by showing that the traction-driven problem is indeed well-posed.
We also take a look at some results on higher regularity as this is a vital part for all further analysis.
After introducing the parameter-to-state map we establish important properties of the forward operator that are crucial for the numerical treatment of the inverse problem.
We prove Fréchet-differentiability and characterize the derivative as the solution of a coupled system of linear differential equations.
The adjoint operator of the linearized forward problem will also be established by its characterization of another system of differential equations.
\subsection{Well-Posedness}
One of the main results in this work is the proof of well-posedness in the sense of Hadamard for the traction-driven problem.
We state this fact in form of the following theorem.
  \begin{Theorem}\label{theo:direct_problem_well-posed} 
    Suppose the assumptions made in Definition \ref{def:direct_problem} hold.
    Then there exists a unique solution $\displacements$ to the traction-driven problem, which depends Lipschitz continuously on given data $\forces$, $\stressforces$, $\damage_0$, and $\damageprocess$ according to \eqref{eq:data_pts_map}.%
  \end{Theorem}%
  To prove Theorem \ref{theo:direct_problem_well-posed} we decouple momentum balance from damage evolution.
  By treating these subproblems individually and establishing their well-posedness in \Cref{prop:DaEvEq_well-posed,prop:EqOfMo_well-posed}, respectively, we lay the groundwork for employing Banach's fixed-point theorem to show existence and uniqueness of a solution to the coupled problem.
 Its Lipschitz-continuous dependence on given data can be shown by employing the individually verifyed properties for the decoupled problems.\\[2ex]
  The previous results then allow us to define a solution operator for the coupled problem mapping the data to the respective solution in Corollary \ref{cor:solution_operator_nonlinear_problem}.
\subsubsection{Damage Evolution}
We start with the analysis of the damage evolution need and propose the following.
  \begin{Proposition}[Well-posedness of decoupled damage evolution]\label{prop:DaEvEq_well-posed}%
    Provided that $\damage_0\in\damagefunctions_0$ and $y\in\integrablefunctions^{\infty}(\timedomain;\integrablefunctions^{\infty}(\spatialdomain)$ with $0\le y\le \timedomainmax^{-1}(\damageconstant_1-\damageconstant_0)(1-\damageconstant_1)^\alpha$ almost everywhere hold, the decoupled damage evolution%
    \begin{subequations}
      \begin{alignat}{2}\label{eq:dc_DaEvEq}
        \damage'%
        &=\left(%
        1-\damage%
        \right)^{-\alpha}%
        y %
        &\quad\tn{in }&\integrablefunctions^{\infty}(\timedomain;\integrablefunctions^{\infty}(\spatialdomain)), \\ %
	\label{eq:dc_DaEvEq_initial_condition} %
        \damage(0)&=\damage_0  &\quad\tn{in }&\integrablefunctions^{\infty}(\spatialdomain),
      \end{alignat}
    \end{subequations}
    is uniquely solved by the damage function $\damage\in\damagefunctions$ which Lipschitz-continuously depends on the data, i.e.,
    \begin{equation}\label{eq:damage_is_Lipschitz}
        \|\damage_1-\damage_2\|_{%
          \sobolevfunctions^{1,\infty}(\timedomain;\integrablefunctions^{\infty}(\spatialdomain))%
        }
        \le c\left(
        \|\damage_{10}-\damage_{20}\|_{\integrablefunctions^\infty(\spatialdomain)} + \|y_1-y_2\|_{\integrablefunctions^{\infty}(\timedomain;\integrablefunctions^{\infty}(\spatialdomain))}
        \right)
    \end{equation}
    holds for some constant $c>0$ and admissible data.
  \end{Proposition}
  \begin{proof}
    The proof is done in four steps. In%
  \begin{inparaenum}[(a)]
    \item\!\!, we treat the spatial variable as a parameter and show that for almost every $\spatialvariables\in\spatialdomain$ exists a function $\omega\in\continuousfunctions^{0,1}(\overline{\timedomain};\,[0,\damageconstant_1])$ such that
    \begin{equation}\label{eq:DaEvEq_integrated}
      \omega(\timevariable)%
      =\damage_0(\spatialvariables) %
      +\int_0^\timevariable(1-\omega(\xi))^{-\alpha} %
      y(\xi,\spatialvariables)\ \!d\xi %
    \end{equation}
    holds. In \item\!\!, we then construct a time- and space-dependent damage function using the previous results and show that this one solves the decoupled damage evolution problem. We establish validity of inequality \eqref{eq:damage_is_Lipschitz} in \item.\\[2ex]
  \end{inparaenum}
  \begin{inparaenum}[(a)]
   \item Let $\omega\in\continuousfunctions(\overline{\timedomain};[0,\damageconstant_1])$ with a fixed $\spatialvariables\in\spatialdomain$ such that inequalities \eqref{eq:initial_damage} for $\damage_0$ and \eqref{eq:damage_sources} for $\damageprocess$ hold. We introduce
   \begin{equation}\label{eq:defining_O}
     (O_{\spatialvariables}\omega)(\timevariable):=\damage_0(\spatialvariables) + \int_0^\timevariable(1-\omega(\xi))^{-\alpha}y(\xi,\spatialvariables)\ \!d\xi
   \end{equation}
   and propose that this induces an operator $O_{\spatialvariables}\colon\continuousfunctions(\overline{\timedomain};[0,\damageconstant_1])\to\continuousfunctions(\overline{\timedomain};[0,\damageconstant_1])$ for almost all $\spatialvariables\in\spatialdomain$. This holds true, since, firstly, $O_{\spatialvariables}\omega$ is actually Lipschitz-continuous in time because of
   \begin{equation}\label{eq:O_is_Lipschitz_in_time}
     |(O_{\spatialvariables}\omega)(\timevariable_1)-(O_{\spatialvariables}\omega)(\timevariable_2)| %
      \le (1-\damageconstant_1)^{-\alpha}\timedomainmax^{-1}(\damageconstant_1-\damageconstant_0)(1-\damageconstant_1)^{\alpha}|\timevariable_1-\timevariable_2|.
   \end{equation}
  The term is bounded as well for all $\timevariable\in\overline{\timedomain}$ as
   \begin{equation}\label{eq:f-p_operator_bounded}
     0 %
       \le
         \damage_0(\spatialvariables) + \int_0^\timevariable %
         (1-\omega(\xi))^{-\alpha}y(\xi, \spatialvariables)\,d\xi %
         = (O_{\spatialvariables}\omega)(\timevariable) %
       \le \damageconstant_1
   \end{equation}
   shows. Secondly, $O_{\spatialvariables}$ is Lipschitz as for any $\omega_1, \omega_2\in\continuousfunctions(\overline{\timedomain};[0,\damageconstant_1])$
   \begin{equation}\label{eq:O_is_Lipschitz_in_d}
     \Big|(O_{\spatialvariables}\omega_1)(\timevariable)-(O_{\spatialvariables}\omega_2)(\timevariable)\Big| %
       \le c_{\spatialvariables} \|\omega_1-\omega_2\|_{\continuousfunctions(\bar\timedomain)}
   \end{equation}
   holds for an arbitrary $\timevariable\in\overline{\timedomain}$. Introducing the weighted norm
   \begin{equation}\label{eq:Biliecki_norm}
     \|f\|_{\bilieckiconstant} := \max_{\timevariable\in\bar\timedomain}\left(\exp(-\bilieckiconstant\timevariable)|f(\timevariable)|\right)
   \end{equation}for a fixed $\bilieckiconstant>0$ and arguing similarly to \eqref{eq:O_is_Lipschitz_in_d} leads to
   \begin{equation}\label{eq:O_is_Lipschitz_in_lambda-norm}
     \|O_{\spatialvariables}\omega_1-O_{\spatialvariables}\omega_2\|_\bilieckiconstant \le \bilieckiconstant^{-1}c_{\spatialvariables}\|\omega_1-\omega_2\|_\bilieckiconstant
   \end{equation}
  for $c_{\spatialvariables}>0$. Choosing $\bilieckiconstant$ large enough makes $O_{\spatialvariables}$ \emph{k-contractive} on $\continuousfunctions(S;\,[0,\damageconstant_1])$ endowed with $\|\cdot\|_\bilieckiconstant$.
  Employing \emph{Banach's fixed-point theorem} guarantees the existence and uniqueness of a fixed-point $O_{\spatialvariables}\omega_{\spatialvariables}=\omega_{\spatialvariables}\in\continuousfunctions^{0,1}(\timedomain;[0,\damageconstant_1])$ endowed with its standard norm.
  The norm $\|\cdot\|_{\bilieckiconstant}$ is equivalent to the maximum norm on $\continuousfunctions(\bar{\timedomain};[0,\damageconstant_1])$.
  The key element is inequality \eqref{eq:O_is_Lipschitz_in_time}. Since $\spatialvariables\in\spatialdomain$ was chosen arbitrarily, this is true almost everywhere in $\spatialdomain$. We note that $c_{\spatialvariables}$ in \eqref{eq:O_is_Lipschitz_in_d} and \eqref{eq:O_is_Lipschitz_in_lambda-norm} is uniformly bounded as a direct consequence of \eqref{eq:damage_sources} and \eqref{eq:defining_O}.\\[2ex]%
  \item We introduce $\damage(\timevariable,\spatialvariables):=\omega_{\spatialvariables}(\timevariable)$ and note that this defines $\damage(\timevariable,\cdot)$ almost everywhere in $\spatialdomain$.  We immediately see that for almost every $\timevariable\in\timedomain$
   \begin{equation}\label{eq:bounds_damage_functions}
     |\damage(\timevariable,\spatialvariables)| \le \damageconstant_1, \quad %
     |\damage'(\timevariable,\spatialvariables)| %
       \le |\damageconstant_1-\damageconstant_0|\timedomainmax^{-1}
  \end{equation}
   hold almost everywhere in $\spatialdomain$. We introduce the set of Lebesgue-measurable functions from $\spatialdomain$ to a Banach space $Y$ by $\measurablefunctions(\spatialdomain;Y)$ and prove $\damage(\timevariable,\cdot)\in\measurablefunctions(\spatialdomain;\mathbb{R})$ for almost all $\timevariable\in\timedomain$ next. To this end, we take the \emph{sequence of successive approximation}, i.e., we define recursively $\damage_{n+1}(\timevariable,\spatialvariables):=O_{\spatialvariables}\damage_n(\timevariable,\spatialvariables)$ and take the initial damage $\damage_0$ to be the first element in this sequence. Since pointwise limits of measurable functions are also measurable and $\damage_0\in\measurablefunctions(\spatialdomain;\mathbb{R})$ by design, we can infer from \eqref{eq:bounds_damage_functions} that $\damage(\timevariable)\in\integrablefunctions^{\infty}(\spatialdomain)$ holds almost everywhere in time. So far, we have shown that the right-hand side of \eqref{eq:dc_DaEvEq} lies in $\integrablefunctions^{\infty}(\spatialdomain)$ for almost every point in time. Recalling inequality \eqref{eq:O_is_Lipschitz_in_time} together with \emph{Rademacher's theorem} then establishes $\damage'(\timevariable)\in\measurablefunctions(\spatialdomain;\mathbb{R})$.
   Collecting all results then shows $\damage\in\damagefunctions$.\\[2ex]
  \item Subtracting the generalized derivatives $\damage_1',\damage_2'$, a subsequent integration of the results over $[0,\timevariable]$, applying the mean value theorem and utilizing \emph{Gronwall's lemma} yield
  \begin{equation}\label{eq:DaEvEq_estimate_1}
    \|\damage_1-\damage_2\|_{%
    \integrablefunctions^\infty(\timedomain;\,\integrablefunctions^\infty(\spatialdomain)) %
    } \\
    \le c\bigg(
      \|\damage_{10}-\damage_{20}\|_{\integrablefunctions^\infty(\spatialdomain)} + \|y_1-y_2\|_{\integrablefunctions^{\infty}(\timedomain;\integrablefunctions^{\infty}(\spatialdomain))}
    \bigg)
  \end{equation}
  for some generic positive constant $c$.
  This inequality is later used to show well-posedness of the forward operator. Then, time and space dependent displacement gradients will take the place of $\mbf{y_1},\mbf{y_2}$. %
  Based on \eqref{eq:dc_DaEvEq} and the inequalities satisfied by $\damage\in\damagefunctions$ we see that
  \begin{equation}\label{eq:d_continuously_depends_on_data}
    |\damage_1'(\timevariable,\spatialvariables)-\damage_2'(\timevariable,\spatialvariables)| %
    \le c\bigg(%
      |\damage_1(\timevariable,\spatialvariables) - \damage_2(\timevariable,\spatialvariables)| %
     + \|y_1-y_2\|_{\integrablefunctions^{\infty}(\timedomain;\integrablefunctions^{\infty}(\spatialdomain))} %
      \bigg)
  \end{equation}
  holds for some constant $c>0$. Combining the last two inequalities completes the proof.
  \end{inparaenum}
  \end{proof}
In view of later sections, particularly when showing that the inverse problem is ill-posed in Sections \ref{sec:Ill_Posedness_of_Nonlinear_Inverse_Problem} and \ref{sec:strong_tangential_cone_condition}, we need to establish stronger regularity results than the previous one.
\begin{Proposition}\label{prop:DaEvEq_higher_regularity}
    Provided the conditions from Proposition \ref{prop:DaEvEq_well-posed} are met and, additionally, $\damage_0\in\damagefunctions_0\cap\sobolevfunctions^{l,\infty}(\spatialdomain)$ and $y\in\sobolevfunctions^{k,\infty}(\timedomain;\sobolevfunctions^{l,\infty}(\spatialdomain))$ hold for $k,l\ge 0$. Then the solution exhibits higher regularity, i.e., $\damage\in\damagefunctions\cap\sobolevfunctions^{k+1,\infty}(\timedomain;\sobolevfunctions^{l,\infty}(\spatialdomain))$ and the estimate
    \begin{equation}\label{eq:damage_is_Lipschitz_in_stronger_norm}
      \|\damage_1-\damage_2\|_{%
          \sobolevfunctions^{k+1,\infty}(\timedomain;\sobolevfunctions^{l,\infty}(\spatialdomain))%
        }
        \le \constant\left(
        \|\damage_{10}-\damage_{20}\|_{\sobolevfunctions^{l,\infty}(\spatialdomain)} + \|y_1-y_2\|_{\sobolevfunctions^{k,\infty}(\timedomain;\sobolevfunctions^{l,\infty}(\spatialdomain))}
        \right)
    \end{equation}
    is valid.
\end{Proposition}
  \begin{proof}
    We will prove this statement via an induction argument.
    The case $k=l=0$ is provided by the previous proposition, i.e., Proposition \ref{prop:DaEvEq_well-posed}.
    Without loss of generality we focus on the case $k=l=1$ for the induction step.\\[2ex]
    Note that $x\mapsto x^{-\alpha}\in\continuousfunctions^{\infty}_0([1-\damageconstant_1,1])$ holds.
    We start by showing higher regularity in time.
    To this end we formally differentiate the right-hand side of \eqref{eq:dc_DaEvEq} weakly with respect to time by employing product- and chain rule for Sobolev functions yielding
    \begin{equation}\label{eq:DaEvEq_higher_regularity_diff_rhs}
        \damage''=\left(\left(1-\damage\right)^{-\alpha}y\right)' = \alpha(1-\damage)^{-(\alpha +1)}\damage'\cdot y + (1-\damage)^{-\alpha}y'.
    \end{equation}
    Following from Proposition \ref{prop:DaEvEq_well-posed} and the imposed conditions on $y$, we can verify that this equality holds true in $\integrablefunctions^{\infty}(\timedomain;\integrablefunctions^{\infty}(\spatialdomain))$.
    Since weak derivatives are uniquely determined, we arrive at  $\damage\in\sobolevfunctions^{2,\infty}(\timedomain;\integrablefunctions^{\infty}(\spatialdomain))\cap\damagefunctions$.\\[2ex]
    For spatial regularity we argue as follows.
    We employ the notation introduced in \eqref{def:mollified_gradient} and make use approximate weak derivatives via difference quotients.
    To this end we take $\spatialvariable\in\spatialdomain_0\subset\subset\spatialdomain$ and $0<|h|<\text{dist}(\spatialdomain_0,\partial\spatialdomain)$.
    We start with the difference quotient for the damage variable and make use of its integral representation, i.e.,
    \begin{multline}
        \left|D_i^h\damage\right| = |\frac{1}{h}\damage_h - \damage| \\
        = \left|D_i^h\damage_0 + \frac{1}{h}\int_0^\timevariable\left(\left(1-\damage_h\right)^{-\alpha} - \left(1-\damage\right)^{-\alpha}\right)y_h + \left(\left(1-\damage\right)^{-\alpha}\left(y_h-y\right)\right)\,d\tau\right| \\
        \le C_1\int_0^{\timevariable}\left|\frac{1}{h}(\damage_h-\damage)\right|\,d\tau + C_2\int_0^{\timevariable}\left|\frac{y_h-y}{h}\right|\,d\tau + \left\|D_i^h\damage_0\right\|_{\integrablefunctions^{\infty}(\spatialdomain)} \\
        \le C_1\int_0^{\timevariable}\left|D_i^h\damage\right|\,d\tau + C.
    \end{multline}
    The second integral after the first inequality is bounded, since $y\in\sobolevfunctions^{2,\infty}(\timedomain;\sobolevfunctions^{1,\infty}(\spatialdomain))$.
    Employing Gronwall's lemma establishes $\|D_i^h\|_{\integrablefunctions^{\infty}(\spatialdomain_0)}\le C$ a.e. in $\timedomain$ and by standard arguments on interdependence between weak derivatives and difference quotients (see \cite{LaEv10, GiTr83} or \cite{MaDo10}, e.g.) we see that $\damage(\timevariable, \cdot)\in\sobolevfunctions^{1,\infty}(\spatialdomain)$ holds a.e. in $\timedomain$. \\[2ex]
    Formally differentiating \eqref{eq:dc_DaEvEq} again, but this time weakly with respect to space leads to a very similar equation like \eqref{eq:DaEvEq_higher_regularity_diff_rhs} but where all time derivatives on the right-hand side are replaced by partial ones and the first derivative in time on the left-hand side is as well.
    We can directly argue that the right-hand side lies in $\integrablefunctions^{\infty}(\spatialdomain)$. Similar arguments like in the beginning then reveal $\damage'(\timevariable, \cdot)\in\sobolevfunctions^{1,\infty}(\spatialdomain)$ a.e. in $\timedomain$.
    The same argument also holds true for $\damage''$, thus $\damage\in\sobolevfunctions^{2,\infty}(\timedomain;\sobolevfunctions^{1,\infty}(\spatialdomain))$. \\[2ex]
    Continuous dependence on the data in stronger norms shown in \eqref{eq:damage_is_Lipschitz_in_stronger_norm} then follows from standard arguments utilizing the structure of the damage variable.
  \end{proof}
\subsubsection{Momentum Balance}
  In the next part, we look at well-posedness of the decoupled momentum balance equation.
\begin{Proposition}[Well-posedness of the equation of motion]\label{prop:EqOfMo_well-posed}
  Provided that
  \begin{equation}\label{eq:EqOfMo_conditions_on_data}
    \damage\in\integrablefunctions^{\infty}(\timedomain;\integrablefunctions^{\infty}(\spatialdomain)), \quad %
    \forces\in \integrablefunctions^\infty(\timedomain;\,\spatialfunctions^*), \quad %
    \stressforces\in \integrablefunctions^\infty(\timedomain;\,\boundaryfunctions^*),
  \end{equation}
  where $\damage$ satisfies the bounds from \eqref{eq:damage_functions}, the traction-driven subproblem
  \begin{equation}\label{eq:EqOfMo_abstract_equation_1}
      \mathcal{A}(\damage)\displacements = \forces + \stressforces, \quad\text{in }\integrablefunctions^{\infty}(\timedomain;\spatialfunctions^*)
  \end{equation}
  is uniquely solvable.
  Furthermore, its solution $\displacements\in\integrablefunctions^\infty(\timedomain;\spatialfunctions)$ is Lipschitz-continuous with respect to the given data, i.e.,
  \begin{equation}\label{eq:dc_EqOfMo_Lipschitz}
    \begin{aligned}
      \|\displacements_1&-\displacements_2\|_{\integrablefunctions^\infty(\timedomain;\spatialfunctions)} \\
        &\le c \Big(\|\forces_1-\forces_2\|_{\integrablefunctions^\infty(\timedomain;\spatialfunctions^*)} + \|\stressforces_1-\stressforces_2\|_{\integrablefunctions^\infty(\timedomain;\boundaryfunctions^*)} \\
        &\quad + \|\elasticityoperator\|_{\linearoperators(\spatialfunctions;\spatialfunctions^*)} %
          \Big( \|\forces_2\|_{\integrablefunctions^\infty(\timedomain;\spatialfunctions^*)} + \|\stressforces_2\|_{\integrablefunctions^\infty(\timedomain;\boundaryfunctions^*)}\Big)\|\damage_1-\damage_2\|_{\integrablefunctions^{\infty}(\timedomain; \integrablefunctions^\infty(\spatialdomain))} \Big)
    \end{aligned}
  \end{equation}
  for some constant $c>0$.
  Choosing data with higher regularity in time, i.e.,
  \begin{equation}\label{eq:EqOfMo_data_higher_regularity}
    \damage\in\sobolevfunctions^{r,\infty}(\timedomain;\integrablefunctions^{\infty}(\spatialdomain)),\quad \forces\in\qisobolevfunctions^r(\timedomain;\spatialfunctions^*),\quad \stressforces\in\qisobolevfunctions^r(\timedomain;\boundaryfunctions^*)
  \end{equation}
  immediately entails $\displacements\in\qisobolevfunctions^r(\timedomain;\spatialfunctions)$.
\end{Proposition}
\begin{proof}
Let $\timevariable\in\timedomain$ be fixed such that $0\le\damage(\timevariable)\le\damageconstant_1$ and $\tilde\forces(\timevariable):=\forces(\timevariable) + \stressforces(\timevariable)$ in $\spatialfunctions^*$.
Note that $\mathcal{\elasticityoperator}(\damage)$ induces a symmetric, continuous and coercive bilinear form, where the upper bound $\damageconstant_1$ of $\damage$ and the inequalities of Korn- and Poincar\'e-type ensure its coerciveness (see \cite{PhCi88, DuLi76, EbZe90}). %
Then Lax-Milgram's theorem (see \cite{RaSh97, EbZe90} e.g.) guarantees a unique solution $\displacements(\timevariable)\in\spatialfunctions$ to
\begin{equation}\label{eq:EqOfMo_abstract_equation_2}
  \mathcal{\elasticityoperator}(\damage)\displacements(\timevariable)=\tilde\forces(\timevariable)\quad\textnormal{in }\spatialfunctions^*,
\end{equation}
which depends Lipschitz continuously on $\tilde\forces(\timevariable)$ and satisfies
\begin{equation}\label{eq:EqOfMo_bound_for_displacements}
 \|\displacements(\timevariable)\|_{\spatialfunctions}\le c(\damageconstant_1)\|\tilde\forces(\timevariable)\|_{\spatialfunctions^*}\text{ a.e. in } \timedomain.
\end{equation}
This induces a mapping $\displacements\colon\timedomain\to\spatialfunctions$ almost everywhere in $\timedomain$. We claim that $\displacements\in\integrablefunctions^\infty(\timedomain;\spatialfunctions)$ holds. \\[2ex]
We begin by showing $\displacements\in\measurablefunctions(\timedomain;\spatialfunctions)$. Since $\tilde\forces\in\measurablefunctions(\timedomain;\spatialfunctions^*)$ and $\damage\in\measurablefunctions(\timedomain;\integrablefunctions^\infty(\spatialdomain))$, there exist sequences of simple functions $\tilde\forces_j\in\simplefunctions(\timedomain;\spatialfunctions^*)$ and $\damage_j\in\simplefunctions(\timedomain;\integrablefunctions^\infty(\spatialdomain))$ converging towards $\tilde\forces$ in $\spatialfunctions^*$ and $\damage$ in $\integrablefunctions^\infty(\spatialdomain)$, respectively.
As a consequence of Lax-Migram's theorem, the existence of a sequence of simple functions $\displacements_j(\timevariable):=\mathcal{\elasticityoperator}^{-1}(\damage_j)\tilde\forces_j(\timevariable)\in\simplefunctions(\timedomain;\spatialfunctions)$ is guaranteed, which converges towards $\displacements$ in $\spatialfunctions$ almost everywhere in $\timedomain$ for $j\to\infty$.
This last part follows from
\begin{equation}\label{eq:convergence_of_simple_functions}
  \|\displacements(\timevariable)-\displacements_j(\timevariable)\|_{\spatialfunctions}
  \le\|(\elasticityoperator_{\damage}^{-1}(\timevariable)-\elasticityoperator_{\damage_j}^{-1}(\timevariable))\|_{\linearoperators(\spatialfunctions^*;\spatialfunctions)}\|\tilde\forces(\timevariable)\|_{\spatialfunctions^*}
  + \|\elasticityoperator_{\damage}^{-1}(\timevariable)\|_{\linearoperators(\spatialfunctions^*;\spatialfunctions)}\|(\tilde\forces(\timevariable)-\tilde\forces_j(\timevariable))\|_{\spatialfunctions^*}
\end{equation}
thus proving $\displacements\in\measurablefunctions(\timedomain;\spatialfunctions)$.
Therefore, $\displacements\in\integrablefunctions^\infty(\timedomain;\spatialfunctions)$ holds as a direct consequence of \eqref{eq:EqOfMo_abstract_equation_2} and $\tilde\forces\in\integrablefunctions^\infty(\timedomain;\spatialfunctions^*)$. Lipschitz continuity with respect to the given data can be established by subtracting the respective equations \eqref{eq:EqOfMo_abstract_equation_2} for two given data sets and by testing the resulting equation with the difference of the respective solutions.
Since time is merely a parameter here, too, our claim on higher regularity immediately follows from \eqref{eq:EqOfMo_abstract_equation_1} and the fact that the damage variable as well as the right-hand side show the prescribed regularity in time.
\end{proof}
In order to present the damage evolution equation while replacing $\boldsymbol{y}$ with $\nemytskiioperator(\mollifiedgradient\displacements)$ we need to make sure that $\mollifiedgradient\displacements$ is essentially bounded.
To this end we have look at results on higher regularity and integrability of the displacements' gradient. \\[2ex]
Imposing slightly more regularity on the right-hand side of the momentum balance equation results in a slightly better space for the solution, i.e., $\displacements\in\spatialfunctions\cap\sobolevfunctions^{1,p}(\spatialdomain)$ for some $p>2$.
This higher integrability has been in the focus of many researchers.
We refer to \cite{Ha-DiJoKnRe16, Ha-DiRe09, HeMe11, ShiWri99}.
\begin{Proposition}[Higher Integrability]\label{prop:EqOfMo_higher_int}
  Assuming the conditions of \Cref{prop:EqOfMo_well-posed} are met. Then there exists a $\bar{p}>2$ such that for all let $p\in[2,\bar{p}]$ and
  \begin{equation}
      \forces\in\integrablefunctions^{\infty}(\timedomain;\sobolevfunctions_{\boundary_0}^{-1,p}(\spatialdomain)^{\dimension}), \quad \stressforces\in\integrablefunctions^{\infty}(\timedomain;\sobolevfunctions^{-(1-\frac{1}{p}),p}(\boundary_1)^{\dimension}).
  \end{equation}
  the traction-driven problem is uniquely solvable with $\displacements\in\integrablefunctions^{\infty}(\timedomain;\sobolevfunctions^{1,p}(\spatialdomain)^{\dimension})$, which depends Lipschitz-continuously on the data, i.e.,
  \begin{multline}
       \|\displacements_1-\displacements_2\|_{\integrablefunctions^{\infty}(\timedomain;\sobolevfunctions^{1,p}(\spatialdomain)^{\dimension})} \\
          \le c \Big(\|\forces_1-\forces_2\|_{\integrablefunctions^{\infty}(\timedomain;\sobolevfunctions_{\boundary_0}^{-1,p}(\spatialdomain)^{\dimension})} + \|\stressforces_1-\stressforces_2\|_{\integrablefunctions^{\infty}(\timedomain;\sobolevfunctions^{-(1-\frac{1}{p}),p}(\boundary_1)^{\dimension})} \\
          + \|\damage_1-\damage_2\|_{\integrablefunctions^{\infty}(\timedomain; \integrablefunctions^\infty(\spatialdomain))} \Big).
  \end{multline}
\end{Proposition}
\begin{proof}
    As mentioned in the proof of Proposition \ref{prop:EqOfMo_well-posed}, time is merely a parameter and after having shown the higher integrability, we can argue in the same manner.
    The proof for the higher integrability can be found in the respective references stated before.
\end{proof}
When investigating the inverse problem in Section \ref{sec:Inverse_Problem} we are in need for results on higher differentiability to stay in a Hilbertspace setting.
To this end, we state the following result which can be found in \cite{PhCi88, PiGr11}.
\begin{Proposition}[Higher Differentiability]\label{prop:EqOfMo_higher_diff}
  Let $k\in\mathbb{N}$ be any positive integer.
  In addition, we assume that the boundaries $\boundary_0, \boundary_1$ of the spatial domain $\spatialdomain$ are of class $C^{k+1,1}$.
  Provided that
  \begin{align} \nonumber
      \elasticitytensor_{ijkl}\in\sobolevfunctions^{k,\infty}(\spatialdomain),
      &\ \damage\in\damagefunctions\cap\sobolevfunctions^{1,\infty}(\timedomain;\sobolevfunctions^{k,\infty}(\spatialdomain)),\\
      \forces\in\integrablefunctions^{\infty}\left(\timedomain;\sobolevfunctions^{k,p}(\spatialdomain)^{\dimension}\right),
      &\ \stressforces\in\integrablefunctions^{\infty}\left(\timedomain;\sobolevfunctions^{k-1-\frac{1}{p},p}(\spatialdomain)^{\dimension}\right)
  \end{align}
  holds, there exists a unique solution $\displacements\in\integrablefunctions^{\infty}\left(\timedomain;\sobolevfunctions^{k+2,p}(\spatialdomain)^{\dimension}\right)$ that depends continuously on the given data.
\end{Proposition}
\begin{Corollary}\label{cor:Operators_decoupled}
  Propositions \ref{prop:DaEvEq_well-posed},\ref{prop:EqOfMo_well-posed}, Lemma \ref{lem:nem-op_well-posed}, and Proposition \ref{prop:EqOfMo_higher_int} allow for the introduction of the two Lipschitz operators $O_{E}^{\forces+\stressforces}\colon\damagefunctions\to\integrablefunctions^{\infty}(\timedomain;\spatialfunctions)$ and $O_{D}\colon\domain(O_D)\subset\integrablefunctions^{\infty}(\timedomain;\integrablefunctions^{\infty}(\spatialdomain))\to\damagefunctions$.
  The superscript indicate respective right-hand sides of the specific equations the operators are linked to.
  If omitted, the right-hand sides are those used in \Cref{prop:DaEvEq_well-posed} or \Cref{prop:EqOfMo_well-posed}, respectively.
  In the decoupled setting $O_E$, $O_D\circ\nemytskiioperator\circ\mollifiedgradient$ map a given damage function to respective displacements and vice versa.
\end{Corollary}
\begin{proof}
  To clarify that $O_D\circ\nemytskiioperator\circ\mollifiedgradient$ is actually well-defined, we show that $\mollifiedgradient\displacements(\timevariable)$ is uniformly bounded in $\spatialdomain$, i.e., $\mollifiedgradient\displacements(\timevariable)\in\overline{Y}=\overline{\mathbb{B}}(0;\bar y)\subset\mathbb{R}^{\dimension^2}$ for almost every $\timevariable\in\timedomain$.
  This immediately follows from \eqref{eq:EqOfMo_bound_for_displacements} and \Cref{prop:EqOfMo_higher_int}.
  Choosing $\bar y$ accordingly then shows our claim.
\end{proof}
\begin{Remark}
  Recall, the domain of $O_D$ denoted by $\domain(O_D)$ are all functions $y\in\integrablefunctions^{\infty}(\timedomain;\integrablefunctions^{\infty}(\spatialdomain))$ satisfying the bounds given in \Cref{prop:DaEvEq_well-posed}.
\end{Remark}
\subsubsection{Coupled Problem}\label{subsec:coupled_problem}
To prove that the parameter-to-state map is well-defined in the suggested setting, we have to make sure, that an estimate in the shape of
\begin{equation}\label{eq:bound_L_infty_by_L_2}
\|\nemytskiioperator(\nabla\displacements_1)(\timevariable)-\nemytskiioperator(\nabla\displacements_2)(\timevariable)\|_{\integrablefunctions^\infty(\spatialdomain)} %
\le c\|\displacements_1(\timevariable)%
-\displacements_2(\timevariable)\|_{\spatialfunctions} %
\text{ a.e. in }\timedomain
\end{equation}
holds.
This means, higher regularity of the displacements $\displacements$ is needed to employ embedding theorems.
Although recent results in \cite{Ha-DiJoKnRe16, HeMe11} state higher integrability of the displacements, particularly $p>2$, this is, unfortunately, not enough to ensure an estimate like \eqref{eq:bound_L_infty_by_L_2}.
Another approach would be to assume smoother data and therefore to achieve higher regularity locally by standard $\integrablefunctions^2$-theory, see \cite{LaEv10, GiTr83, LouNi55} for instance.
But this does not remedy the situation as the right-hand side would be measured in a stronger norm which cannot be bounded by a damage function in its respective norm.\\[2ex]
Another option could be to change into an $\integrablefunctions^{p}$ setting.
Especially for the damage evolution in space this could potentially remedy the situation.
However, it is not so clear how this would affect the momentum balance equation in \eqref{eq:EqOfMo} and, additionally, the complexity and difficulty for the inverse problem being now in a general Banach setting would complicate things even further.
That is why, for now, we want to remain in a Hilbert space setting and focus on showing the validity of our approach.
That is why we decided to use a mollified gradient in our model.\\[2ex]
Please note that by introducing additional regularization into the forward problem we affect the singular value decomposition of the linearized operator.
Additional smoothness increases oscillation effects in the reconstruction process (see e.g. \cite{EnHaNeu96, AnRi03}).
That is why we want to regularize as little as possible. \\[2ex]
In order to state the coupled problem in a meaningful way (see Corollary \ref{cor:Operators_decoupled}) and, especially, use them to construct a fixed-point operator on $\damagefunctions$ to prove unique solvability of the fully coupled problem, we need the following lemma.
\begin{Lemma}\label{lem:g_is_Lipschitz}
  Let $\displacements_1(\timevariable), \displacements_2(\timevariable)\in\spatialfunctions\cap\sobolevfunctions^{1,p}(\spatialdomain)$ a.e. in $\timedomain$.
  We denote the mollified gradient by $\mollifiedgradient$.
  Provided that $\damageprocess\in\damageprocesses$ is an admissible damage process, there exists a constant $c(\dimension)>0$ particularly dependent on the spatial dimension $\dimension$ such that
  \begin{equation}\label{eq:Lipschitz_estimate_for_nem-op}
    \|\nemytskiioperator(\mollifiedgradient\displacements_1)(\timevariable)-\nemytskiioperator(\mollifiedgradient\displacements_2)(\timevariable)\|_{\integrablefunctions^\infty(\spatialdomain)} %
    \le c\|\displacements_1(\timevariable)%
    -\displacements_2(\timevariable)\|_{\sobolevfunctions^{1,p}(\spatialdomain)} %
  \end{equation}
  almost everywhere in $\timedomain$.
  \end{Lemma}
\begin{proof}
  We limit ourselves to show this for a mollified gradient $\mollifiedgradient$ in the shape of \eqref{def:mollified_gradient}.
  For spatial dimension $\dimension=1$ the claim follows immediately from Lipschitz-continuity of damage processes and the compact embedding $\qisobolevfunctions^s(\spatialdomain) \hookrightarrow\hookrightarrow\continuousfunctions(\bar\spatialdomain)$ since $k>N/p$ for $p>1$ (see e.g.~\cite[Theorem III.5.8.3]{KuJo77}). \\[2ex]
  For $\dimension=2$ we also use the Lipschitz-continuity of damage processes as before.
  By employing Proposition \ref{prop:EqOfMo_higher_int} we are able to establish \eqref{eq:Lipschitz_estimate_for_nem-op} for some $p>2$.\\[2ex]
  In the three-dimensional case $\dimension=3$, we need more regularization of the displacement's gradient.
  One option would be to taking the average of the displacements in a small ball around $\spatialvariables$ before applying the difference quotient, i.e.,
  \begin{equation}
      \mollifiedgradient\displacements:=D_i^{\mu}\left(\chi_{\mathbb{B}(0;\mu)}\convolution\displacements\right).
  \end{equation}
  We also refer to Remark \ref{rem:on_mollifiers}.
  Then the argument is the same as in the case $\dimension=2$.
\end{proof}
\begin{Remark}\label{rem:mollified_gradient}
    Another possibility to show an estimate like \eqref{eq:Lipschitz_estimate_for_nem-op} would be to use the higher regularity results from Proposition \ref{prop:EqOfMo_higher_diff} since we can choose $p>3$ and can make sure an embedding into continuous functions is possible.
    The drawback is that we have to make stronger assumptions on the elasticitytensor $\elasticitytensor$, the damage functions, and the damage processes as well.
    The left-hand side of estimate \eqref{eq:Lipschitz_estimate_for_nem-op} would then be measured in the stronger norm of $\sobolevfunctions^{1,\infty}(\spatialdomain)$.
\end{Remark}
We can now collect the previous results to make a sound argument for the well-posedness of the fully coupled problem.
\begin{proof}[Proof of Theorem \ref{theo:direct_problem_well-posed}]
  The main idea of this proof is laid out in \Cref{fig:sketch_of_proof}. We start with an arbitrary damage function $\damage_1\in\damagefunctions$. According to \Cref{prop:EqOfMo_well-posed}, we can employ this damage function into the momentum balance equation and solve for a unique displacement field $O_E(\damage_1)=:\displacements_1\in\integrablefunctions^\infty(\timedomain;\spatialfunctions)$. We apply the mollified gradient to $\displacements_1$ and insert $\mollifiedgradient\displacements_1$ into the damage evolution equation. \Cref{prop:DaEvEq_well-posed} established that there is exactly one damage function $\damage_2:=O_D\circ\nemytskiioperator\circ\mollifiedgradient(\displacements_1)$ in $\damagefunctions$ solving equations \eqref{eq:dc_DaEvEq}, \eqref{eq:dc_DaEvEq_initial_condition}. This strategy introduces an operator $\fixedpointoperator\colon\damagefunctions\to\damagefunctions$ via $\fixedpointoperator:=O_D\circ\nemytskiioperator\circ\mollifiedgradient\circ O_E$ for which we claim the existence of a unique fixed-point $\damage$. To see this, we note that $\damagefunctions$ is a complete metric space and show that $\fixedpointoperator^2:=\fixedpointoperator\circ\fixedpointoperator$ is k-contractive.
  Employing Banach's contraction principle then shows our claim.
  \begin{center}
    \includegraphics[width=.75\linewidth]{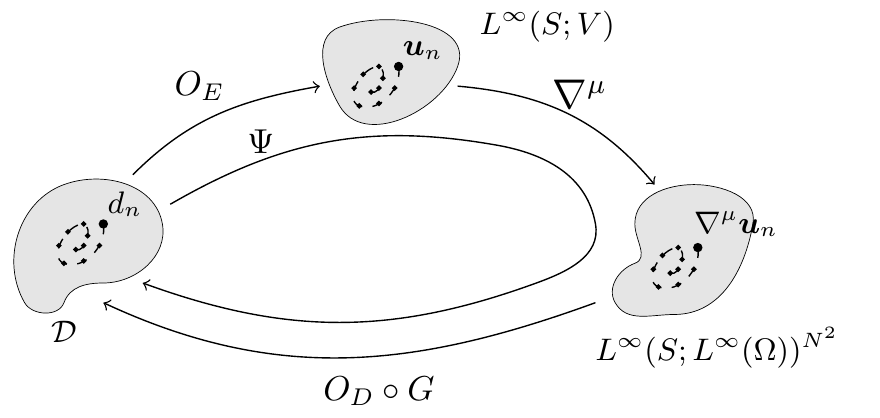}
    \captionof{figure}{Concept to prove Theorem \ref{theo:direct_problem_well-posed}. \small{The operators $O_E$, $O_D$ introduced in Corollary \ref{cor:Operators_decoupled} give raise to constructing a k-contractive mapping $\fixedpointoperator\colon\damagefunctions\to\damagefunctions$ that complies with \emph{Banach's contraction principle}.}}
    \label{fig:sketch_of_proof}
  \end{center}
  Take $\damage_1, \damage_2\in\damagefunctions$. Definition of $\fixedpointoperator$ immediately entails
  \begin{equation}\label{eq:fi-po_op_1}
    {\fixedpointoperator^{2}}'(\damage_i) = (1-\fixedpointoperator^2(\damage_i))^{-\alpha}\nemytskiioperator(\mollifiedgradient O_E(\fixedpointoperator(\damage_i)))\text{ in }\integrablefunctions^{\infty}(\timedomain;\integrablefunctions^{\infty}(\spatialdomain))\text{ for }i=1,2,
  \end{equation}
  and leads us to
  \begin{multline}\label{eq:fi-po_op_estimate_1}
    \|\fixedpointoperator^2(\damage_1)-\fixedpointoperator^2(\damage_2)\|_{\sobolevfunctions^{1,\infty}(\timedomain;\integrablefunctions^\infty(\spatialdomain))}
      \le c\big(\|\fixedpointoperator^2(\damage_1)-\fixedpointoperator^2(\damage_2)\|_{\integrablefunctions^{\infty}(\timedomain;\integrablefunctions^{\infty}(\spatialdomain))} \\
      + \|\nemytskiioperator(\mollifiedgradient(O_E(\damage_1)))-\nemytskiioperator(\mollifiedgradient(O_E(\damage_2)))\|_{\integrablefunctions^{\infty}(\timedomain;\integrablefunctions^{\infty}(\spatialdomain))}\big).
  \end{multline}
  We are left to estimate the right-hand side by $\tilde{c}\|\damage_1-\damage_2\|_{\damagefunctions}$ for some positive $\tilde{c}<1$.
  To this end, we focus on the first term. Integrating \eqref{eq:fi-po_op_1} over $(0,\timevariable)$ shows
    \begin{equation}\label{eq:fi-po_op_estimate_2}
      \|\fixedpointoperator_1^2(\timevariable) - \fixedpointoperator_2^2(\timevariable)\|_{\integrablefunctions^\infty(\spatialdomain)}
        \le c\int_0^t\|\nemytskiioperator(\mollifiedgradient(O_E(\fixedpointoperator(\damage_1))))(\xi)-\nemytskiioperator(\mollifiedgradient(O_E(\fixedpointoperator(\damage_2))))(\xi)\|_{\integrablefunctions^\infty(\spatialdomain)}\,d\xi
    \end{equation}
    for $\timevariable\in\timedomain$. Lemma \ref{lem:g_is_Lipschitz}, Lipschitz continuity of $\nemytskiioperator$ and $O_E$, and using \eqref{eq:fi-po_op_estimate_2} add up to
    \begin{multline}\label{eq:fi-po_op_estimate_3}
      \|\nemytskiioperator(\mollifiedgradient(O_E(\fixedpointoperator(\damage_1))))(\timevariable)-|\nemytskiioperator(\mollifiedgradient(O_E(\fixedpointoperator(\damage_2))))(\timevariable)\|_{\integrablefunctions^\infty(\spatialdomain)} \\
       \le c\|\fixedpointoperator(\damage_1)(\timevariable)-\fixedpointoperator(\damage_2)(\timevariable)\|_{\integrablefunctions^{\infty}(\spatialdomain)}
       \le c\int_0^t\|\damage_1(\xi)-\damage_2(\xi)\|_{\integrablefunctions^\infty(\spatialdomain)}\,d\xi.
    \end{multline}
    Using the $\|\cdot\|_{\bilieckiconstant}$-norm from \eqref{eq:Biliecki_norm} again allows for estimates
      \begin{equation}\label{eq:fi-po_op_estimate_4}
        \int_0^\timevariable\int_0^\xi\|\damage_1(\zeta)-\damage_2(\zeta)\|_{\integrablefunctions^\infty(\spatialdomain)}\,d\zeta\,d\xi %
        \le \bilieckiconstant^{-2}\exp(\bilieckiconstant\timevariable)\|\damage_1-\damage_2\|_\bilieckiconstant
    \end{equation}
    as well as
      \begin{equation}\label{eq:fi-po_op_estimate_5}
        \int_0^\timevariable\|\damage_1(\xi)-\damage_2(\xi)\|_{\integrablefunctions^\infty(\spatialdomain)} %
        \le\bilieckiconstant^{-1}\exp(\bilieckiconstant\timevariable)\|\damage_1-\damage_2\|_\bilieckiconstant
    \end{equation}
  in $\timedomain$. Collecting inequalities \eqref{eq:fi-po_op_estimate_1} to \eqref{eq:fi-po_op_estimate_5} and the fact that $\timevariable$ was chosen arbitrarily, we arrive at
    \begin{equation}\label{eq:fi-po_op_is_k-contractive}
      \|\fixedpointoperator^2(\damage_1)-\fixedpointoperator^2(\damage_2)\|_\bilieckiconstant %
      \le c(\bilieckiconstant^{-1} + \bilieckiconstant^{-2})\|\damage_1-\damage_2\|_\bilieckiconstant.
    \end{equation}
    Choosing $\bilieckiconstant$ sufficiently large completes the proof. Lipschitz continuous dependence on the data follows from the respective results for the subproblems in \Cref{prop:DaEvEq_well-posed} and \Cref{prop:EqOfMo_well-posed}.
    Without loss of generality we only vary $\damageprocess_i\in\damageprocesses$ for $i=1,2$ in the direct problem \eqref{eq:direct_problem} and denote the resulting displacements and damage variables with $\displacements_i$, $\damage_i$, respectively.
    Subtracting the respective equations we get
    \begin{equation}\label{eq:continuous_dependence_for_DaEvEq}
      \|\damage_1(\timevariable) - \damage_2(\timevariable)\|_{\integrablefunctions^\infty(\spatialdomain)} \le C\Big( \|\damageprocess_1-\damageprocess_2\|_{\integrablefunctions^\infty(\timedomain;\integrablefunctions^\infty(\spatialdomain;\continuousfunctions^{0,1}(\overline{Y})))} + \int_0^{\timevariable}\|\displacements_1-\displacements_2\|_{\spatialfunctions} \,ds\Big)
    \end{equation}
    by common manipulations and employment of Lemma \ref{lem:g_is_Lipschitz} (cf.~to \eqref{eq:fi-po_op_estimate_2}).
    In connection with a respective energy estimate from the equation of motion this results in
    \begin{equation}\label{eq:continuous_dependence_for_EqOfMo}
      \|\displacements_1(\timevariable)-\displacements_2(\timevariable)\|_{\spatialfunctions}
        \le C\|\damage_1(\timevariable)-\damage_2(\timevariable)\|_{\integrablefunctions^\infty(\spatialdomain)}
        \le C\Big(\|\damageprocess_1-\damageprocess_2\|_{\integrablefunctions^\infty(\timedomain;\integrablefunctions^\infty(\spatialdomain;\continuousfunctions^{0,1}(\overline{Y})))} + \int_0^{\timevariable}\|\displacements_1-\displacements_2\|_{\spatialfunctions} \,ds\Big)
    \end{equation}
    and another application of Gronwall's lemma then establishes
    \begin{equation}\label{eq:estimate_displacments_via_damageprocesses}
      \|\displacements_1(\timevariable)-\displacements_2(\timevariable)\|_{\integrablefunctions^\infty(\timedomain;\spatialfunctions)} %
      \le C\|\damageprocess_1-\damageprocess_2\|_{\integrablefunctions^\infty(\timedomain;\integrablefunctions^\infty(\spatialdomain;\continuousfunctions^{0,1}(\overline{Y})))} %
      \le C\|\damageprocess_1-\damageprocess_2 \|_{\damageprocesses}
    \end{equation}
  which completes our proof.
  \end{proof}
  \begin{Corollary}[Higher regularity for the coupled problem]\label{cor:fp_higher_reg}
    Assuming that the conditions of \Cref{prop:EqOfMo_higher_diff}, \cref{prop:DaEvEq_higher_regularity}, and
    \begin{equation}
        \damageprocess\in\damageprocesses\cap\sobolevfunctions^{k,\infty}(\timedomain;\sobolevfunctions^{l,p}(\spatialdomain;\continuousfunctions^{l+1,1}(\overline{Y})))
    \end{equation}
    are met.
    Also note \Cref{rem:mollified_gradient}.
    Then the solution exhibits higher regularity, i.e.,
    \begin{equation}
        \damage\in\damagefunctions\cap\sobolevfunctions^{k+1,\infty}(\timedomain;\sobolevfunctions^{l,\infty}(\spatialdomain)), \quad
        \displacements\in\integrablefunctions^{\infty}(\timedomain;\spatialfunctions)\cap\sobolevfunctions^{k+1,\infty}\left(\timedomain;\sobolevfunctions^{l+2,p}(\spatialdomain)^{\dimension}\right),
    \end{equation}
    and the solution depends Lipschitz-continuously on the data.
  \end{Corollary}
  \begin{proof}
    This can be proven in the exact same manner as in the proof of \Cref{theo:direct_problem_well-posed}.
    The only thing that has to be updated is the use of stronger norms.
    But we have already seen how we can use mollifiers and embeddings theorems to get the needed estimates.
  \end{proof}
  \begin{Corollary}\label{cor:solution_operator_nonlinear_problem}
    Theorem \ref{theo:direct_problem_well-posed} gives the means to introduce an operator mapping the data Lipschitz continuously to the respective solution, i.e.,
    \begin{equation}\label{eq:solution_opertator_nonlinear_problem}
        \begin{array}{rcl}
         \ptsmap\colon\integrablefunctions^\infty(\timedomain;\spatialfunctions^*)\times \integrablefunctions^\infty(\timedomain;\boundaryfunctions^*)\times\damagefunctions_0\times\damageprocesses & \to & \integrablefunctions^\infty(\timedomain;\spatialfunctions)\times\damagefunctions\\
          (\forces,\stressforces,\damage_0,\damageprocess) & \mapsto & (\displacements,\damage).
        \end{array}
    \end{equation}
  \end{Corollary}
For a clear and brief presentation, the dependency of the parameter-to-state map on $\forces, \stressforces,$ and $\damage_0$ are henceforth omitted in their notation. Thus, we fix $\forces, \stressforces, \boundarydisplacements,\damage_0$ and interpret $\ptsmap$ as a mapping $\damageprocesses \to \integrablefunctions^\infty(\timedomain;\spatialfunctions)\times\damagefunctions$.
\subsection{Fréchet-Differentiability}\label{subsec:fo-op_is_differentiable}
Since the forward operator $\forwardoperator:=I\circ\pi_1\circ\ptsmap$ only consists of bounded linear operators except for $\ptsmap$, differentiability of $\ptsmap$ directly transfers to $\forwardoperator$.
Here, $I$ denotes the continuous embedding from $\integrablefunctions^\infty(\timedomain;\spatialfunctions)$ into $\integrablefunctions^2(\timedomain\times\spatialdomain)^{\dimension}$.
To prove differentiability of the \emph{parameter-to-state} map $\ptsmap\colon\damageprocesses\to\integrablefunctions^\infty(\timedomain;\spatialfunctions)\times\damagefunctions$, we first show in Lemma \ref{lem:nem-op_differentiable} that admissible damage processes $\damageprocess\in\damageprocesses$ indeed lead to differentiable operators.
We follow \cite[cf.~especially to Section 4.3]{FrTr10} in his presentation.
Afterwards, we employ the chain rule to show that $\ptsmap$ and consequently $\forwardoperator$ are indeed differentiable.
\begin{Lemma}\label{lem:nem-op_differentiable}
   Every admissible damage process $\damageprocess\in\damageprocesses$ generates a Fr\'echet-differentiable  Nemytskii operator $\nemytskiioperator\colon\integrablefunctions^\infty(\timedomain;\integrablefunctions^\infty(\spatialdomain)^{\dimension^2})\to\integrablefunctions^\infty(\timedomain;\integrablefunctions^\infty(\spatialdomain))$ via \eqref{eq:def_nem-op_b}, i.e., for all $\boldsymbol{f},\boldsymbol{h}\in\integrablefunctions^\infty(\timedomain;\integrablefunctions^\infty(\spatialdomain))^{\dimension^2}$
    \begin{equation}\label{eq:nem-op_derivative}
      ((\partial\nemytskiioperator(\boldsymbol{f}))\boldsymbol{h})(\timevariable,\spatialvariables) = \partial_{\boldsymbol{y}}\damageprocess\big(\timevariable,\spatialvariable,\boldsymbol{f}(\timevariable,\spatialvariables)\big)(\boldsymbol{h}(\timevariable,\spatialvariables))
    \end{equation}
    holds for almost all $\timevariable\in\timedomain$ and $\spatialvariables\in\spatialdomain$.
\end{Lemma}
\begin{proof}
  In Lemma \ref{lem:nem-op_well-posed} we showed well-posedness of $\nemytskiioperator$. To prove differentiability and particularly \cref{eq:nem-op_derivative} for every $\boldsymbol{f}\in\integrablefunctions^{\infty}(\timedomain;\integrablefunctions^{\infty}(\spatialdomain))^{\dimension^2}$ we have to search for a linear operator $\partial\nemytskiioperator(\boldsymbol{f})\in\linearoperators(\integrablefunctions^\infty(\timedomain;\integrablefunctions^\infty(\spatialdomain)^{\dimension^2});\integrablefunctions^\infty(\timedomain;\integrablefunctions^\infty(\spatialdomain)))$ satisfying said equation as well as
  \begin{equation}\label{eq:def_differentiable}
    \nemytskiioperator(\boldsymbol{f}+\boldsymbol{h})-\nemytskiioperator(\boldsymbol{f})=(\partial\nemytskiioperator(\boldsymbol{f}))\boldsymbol{h} + o(\|\boldsymbol{h}\|_{\integrablefunctions^{\infty}(\timedomain;\integrablefunctions^{\infty}(\spatialdomain))^{\dimension^2}})\quad (\|\boldsymbol{h}\|_{\integrablefunctions^{\infty}(\timedomain;\integrablefunctions^{\infty}(\spatialdomain))^{\dimension^2}}\to 0).
  \end{equation}
  Let $\timevariable\in\timedomain$, $\spatialvariables\in\spatialdomain$ be fixed and $\boldsymbol{f},\boldsymbol{h}\in\integrablefunctions^\infty(\timedomain;\integrablefunctions^\infty(\spatialdomain))^{\dimension^2}$. Together with Lemma \ref{lem:nem-op_well-posed} we have
  \begin{equation}\label{eq:nem-op_Taylor_expansion}
    \nemytskiioperator(\boldsymbol{f}+\boldsymbol{h})(\timevariable,\spatialvariable)-\nemytskiioperator(\boldsymbol{f})(\timevariable,\spatialvariable) %
    = \partial_{\boldsymbol{y}}\damageprocess\big(\timevariable,\spatialvariables,\boldsymbol{f}(\timevariable,\spatialvariables)\big)\boldsymbol{h}(\timevariable,\spatialvariables) + r(\boldsymbol{f},\boldsymbol{h})(\timevariable,\spatialvariables)
  \end{equation}
  where
  \begin{equation}\label{eq:nem-op_Taylor_remainder}
    r(\boldsymbol{f},\boldsymbol{h})(\timevariable,\spatialvariables)
        :=\int_0^1\left(\partial_{\boldsymbol{y}}\damageprocess\big(\timevariable,\spatialvariables,\boldsymbol{f}(\timevariable,\spatialvariables)+s\boldsymbol{h}(\timevariable,\spatialvariables)\big)-\partial_{\boldsymbol{y}}\damageprocess\big(\timevariable,\spatialvariables,\boldsymbol{f}(\timevariable,\spatialvariables)\big)\right)\boldsymbol{h}(\timevariable,\spatialvariables)\,ds
    \le\frac{1}{2}\|\damageprocess\|_{\damageprocesses}\left|\boldsymbol{h}(\timevariable,\spatialvariables)\right|_{\mathbb{R}^{\dimension^2}}^2.
  \end{equation}
  This leads us to
  \begin{equation}\label{eq:nem-op_remainder_of_first_order}
    \|r(\boldsymbol{f},\boldsymbol{h})\|_{\integrablefunctions^\infty(\timedomain;\integrablefunctions^\infty(\spatialdomain))}\le C\|\boldsymbol{h}\|_{\integrablefunctions^\infty(\timedomain;\integrablefunctions^\infty(\spatialdomain))^{\dimension^2}}^2,
  \end{equation}
  thus proving $r(\boldsymbol{f},\boldsymbol{h})=o(\|\boldsymbol{h}\|_{\integrablefunctions^{\infty}(\timedomain;\integrablefunctions^{\infty}(\spatialdomain))^{\dimension^2}})$. We see that $\boldsymbol{h}(\cdot,\cdot)\mapsto\partial_{\boldsymbol{y}}\damageprocess(\cdot,\cdot,\boldsymbol{f}(\cdot,\cdot))\boldsymbol{h}(\cdot,\cdot)$ is linear and continuous in $\integrablefunctions^\infty\big(\timedomain;\integrablefunctions^\infty(\spatialdomain)^{\dimension^2})\big)$ and $\partial_{\boldsymbol{y}}\damageprocess(\timevariable,\spatialvariables,\boldsymbol{f}(\timevariable,\spatialvariables)$ is bounded because of assumptions made on $\partial_{\boldsymbol{y}}\damageprocess$. As limit of measurable functions $\partial_{\boldsymbol{y}}\damageprocess$ is also measurable. Thus we have proven all properties to verify \cref{eq:nem-op_derivative}.
\end{proof}
\begin{Remark}
  Note that local Lipschitz continuity of the first derivative with respect to $\boldsymbol{f}$ would be sufficient to prove differentiability here, i.e., $\damageprocess(\timevariable,\spatialvariables)\in\continuousfunctions^{\,\text{-}1,1}(\mathbb{R})^{\dimension^2}$.%
\end{Remark}
Theorem \ref{theo:derivative_of_pts_map} shows that derivatives of the parameter-to-state map can be characterized by solutions of linear PDE.
As before, we gather some results on differentiability of the decoupled problems.
We start showing differentiability of $O_E$ at $\damage\in\damagefunctions$ in Lemma \ref{lem:O_E_is_differentiable} and $O_D$ at $y\in\integrablefunctions^{\infty}(\timedomain;\integrablefunctions^{\infty}(\spatialdomain))$ in Lemma \ref{lem:O_D_is_differentiable}, afterwards.
We also prepend an auxiliary result on a specific mapping.
\begin{Lemma}\label{lem:O_E_is_differentiable}
  Given the conditions of Proposition \ref{prop:EqOfMo_well-posed} are met, the operator $O_E\colon\damagefunctions\to\integrablefunctions^{\infty}(\timedomain;\spatialfunctions)$ is differentiable with $\partial O_E\colon\damagefunctions\to\linearoperators(\sobolevfunctions^{1,\infty}(\timedomain;\integrablefunctions^{\infty}(\spatialdomain));\integrablefunctions^{\infty}(\timedomain;\spatialfunctions))$. For $\perturbation\in\sobolevfunctions^{1,\infty}(\timedomain;\integrablefunctions^{\infty}(\spatialdomain))$, $\damage\in\damagefunctions$ we can characterize $\partial O_E(\damage)\perturbation$ as the unique solution to
  \begin{equation}\label{eq:O_E_derivative_character}
    \mathcal{\elasticityoperator}(\damage)(\partial O_E(\damage)\perturbation) = -\mathcal{\elasticityoperator}(1+\perturbation)(O_E(\damage))\text{ in }\integrablefunctions^{\infty}(\timedomain;\spatialfunctions^*)
  \end{equation}
  and, hence,
  \begin{equation}\label{eq:O_E_derivative_character_1}
    \partial O_E(\damage)\perturbation = O_E^{\forces(\perturbation)}(\damage)
  \end{equation}
  where $\forces(\perturbation):=-\mathcal{\elasticityoperator}(1+\perturbation)(O_E(\damage))$.
\end{Lemma}
\begin{proof}
  We have to show that for every $\damage\in\damagefunctions$ exists a uniquely determined linear operator $\partial O_E(\damage)$ such that for every $\varepsilon>0$ there is some $c_\varepsilon>0$ so every $\perturbation\in\sobolevfunctions^{1,\infty}(\timedomain;\integrablefunctions^{\infty}(\spatialdomain))$ with $\|\perturbation\|\le c_\varepsilon$ and $\damage+\perturbation\in\damagefunctions$ satisfies
  \begin{equation}\label{eq:def_f-derivative_for_O_E}
    \|O_E(\damage + \perturbation) - O_E(\damage) - \partial O_E(\damage)\perturbation\|_{\integrablefunctions^{\infty}(\timedomain;\spatialfunctions^*)}\le\varepsilon\|\perturbation\|_{\sobolevfunctions^{1,\infty}(\timedomain;\integrablefunctions^{\infty}(\spatialdomain))}
  \end{equation}
  as well as the identity in \eqref{eq:O_E_derivative_character}. At first, we note that there exists a uniquely determined $\boldsymbol{w}_\perturbation\in\integrablefunctions^{\infty}(\timedomain;\spatialfunctions)$ solving $\mathcal{\elasticityoperator}(\damage)(\boldsymbol{w}_\perturbation) = -\mathcal{\elasticityoperator}(1+\perturbation)(O_E(\damage))\text{ in }\integrablefunctions^{\infty}(\timedomain;\spatialfunctions^*)$ because
  \begin{multline}\label{eq:O_E_is_differentiable_estimate_1}
    \esssup_{\timevariable\in\timedomain}\left(\sup_{\spatialtestfunctions\in\spatialfunctions}\left(\frac{\left|\langle\mathcal{\elasticityoperator}(1+\perturbation)(O_E(\damage)),\spatialtestfunctions\rangle\right|}{\|\spatialtestfunctions\|_{\spatialfunctions}}\right)\right) \\
      \le \|\perturbation\|_{\sobolevfunctions^{1,\infty}(\timedomain;\integrablefunctions^\infty(\spatialdomain))}\|\elasticityoperator\|_{\linearoperators(\spatialfunctions;\spatialfunctions^*)}(\|\forces\|_{\integrablefunctions^{\infty}(\timedomain;\spatialfunctions^*)} + \|\stressforces\|_{\integrablefunctions^{\infty}(\timedomain;\boundaryfunctions^*)} ) < \infty
  \end{multline}
  entails $\mathcal{\elasticityoperator}(1+\perturbation)(O_E(\damage))\in\integrablefunctions^\infty(\timedomain;\spatialfunctions^*)$ for every $\damage\in\damagefunctions$ so that we can use a similar reasoning as in Proposition \ref{prop:EqOfMo_well-posed} to guarantee existence and uniqueness of said solution.
  We recall the definition of $\mathcal{\elasticityoperator}$ via \eqref{eq:Operator_A} and notice the right-hand side to be linear in $\perturbation$ which transfers directly to the operator mapping $\perturbation$ to $\boldsymbol{w}_\perturbation$.
  This actually establishes $w_\perturbation=O_E^{\forces(\perturbation)}(\damage)$, too.
  To see \eqref{eq:def_f-derivative_for_O_E} we take take $\forces$ and $\stressforces$ from Proposition \ref{prop:EqOfMo_well-posed} and by adding a zero we get
  \begin{multline}\label{eq:O_E_derivative_eq_1}
    0 = \mathcal{\elasticityoperator}(\damage+\perturbation)O_E(\damage+\perturbation) - \mathcal{\elasticityoperator}(\damage)O_E(\damage) \\
      = \mathcal{\elasticityoperator}(\damage)(O_E(\damage+\perturbation)-O_E(\damage)-\boldsymbol{w}_\perturbation)+ \mathcal{\elasticityoperator}(1+\perturbation)(O_E(\damage+\perturbation)-O_E(\damage))\text{ in }\integrablefunctions^{\infty}(\timedomain;\spatialfunctions^*).
  \end{multline}
  Rearrangement, strong monotonicity of $\mathcal{\elasticityoperator}(\damage)$, and \eqref{eq:O_E_is_differentiable_estimate_1} then reveal
  \begin{multline}\label{eq:O_E_derivative_estimate_2}
    \|O_E(\damage+\perturbation)-O_E(\damage)-\boldsymbol{w}_g\|_{\integrablefunctions^{\infty}(\timedomain;\spatialfunctions)} \\
      \le c(\elasticityoperator)\|\elasticityoperator\|_{\linearoperators(\spatialfunctions;\spatialfunctions^*)}(\|\forces\|_{\integrablefunctions^{\infty}(\timedomain;\spatialfunctions^*)} + \|\stressforces\|_{\integrablefunctions^{\infty}(\timedomain;\boundaryfunctions^*)})\|\perturbation\|_{\sobolevfunctions^{1,\infty}(\timedomain;\integrablefunctions^{\infty}(\spatialdomain))}^2.
  \end{multline}
  Choosing $c_\varepsilon<\varepsilon(c(\elasticityoperator)\|\elasticityoperator\|(\|\forces\|+\|\stressforces\|))^{-1}$ and uniqueness of derivatives thus shows \eqref{eq:def_f-derivative_for_O_E}.
\end{proof}
\begin{Lemma}\label{lem:O_D_is_differentiable}
 Assuming the conditions of Proposition \ref{prop:DaEvEq_well-posed} are satisfied, the operator $O_D\colon\domain(O_D)\to\damagefunctions\subset\sobolevfunctions^{1,\infty}(\timedomain;\integrablefunctions^{\infty}(\spatialdomain))$ is Fréchet-differentiable with $\partial O_D\colon\domain(O_D)\to\linearoperators(\integrablefunctions^{\infty}(\timedomain;\integrablefunctions^{\infty}(\spatialdomain));\damagefunctions)$. Let $\perturbation\in\integrablefunctions^{\infty}(\timedomain;\integrablefunctions^{\infty}(\spatialdomain))$ denote a perturbation of $y\in\domain(O_D)$ such that $y+\perturbation\in\domain(O_D)$. Then $\partial O_D(y)\perturbation$ is characterized by being the unique solution to
  \begin{subequations}\label{eq:O_D_derivative_character}
    \begin{alignat}{2}\nonumber
      (\partial O_D(y)\perturbation)'
      &=\left(1-O_D(y) \right)^{-\alpha}\perturbation \\ \label{eq:O_D_derivative_character_1}
      &\quad +\alpha(1-O_D(y))^{-(\alpha + 1)}y\partial O_D(y)\perturbation&\text{ in }&\integrablefunctions^{\infty}(\timedomain;\integrablefunctions^{\infty}(\spatialdomain)), \\ %
      \label{eq:O_D_derivative_character_2} %
      \partial O_D(y)\perturbation(0)&=0  &\text{ in }&\integrablefunctions^{\infty}(\spatialdomain).
    \end{alignat}
  \end{subequations}
\end{Lemma}
\begin{proof}
  In \eqref{eq:O_D_derivative_character}, we replace $\partial O_D(y)\perturbation$ through $w_\perturbation$ and acknowledge that the resulting problem is uniquely solvable by following the same strategy as in arguing Proposition \ref{prop:DaEvEq_well-posed}.
  The necessary adaptions are straight forward. Clearly, $w_\perturbation$ is linear in $\perturbation$ and we turn to prove $w_\perturbation = \partial O_D(y)\perturbation$ holds true.
  Proposition \ref{prop:DaEvEq_well-posed} and Corollary \ref{cor:Operators_decoupled} entail that $O_D(y+\perturbation)$, $O_D(y)$ solve the respective damage evolution equations.
  Subtracting these equations from one another leads to
  \begin{multline}\label{eq:O_D_derivative_character_3}
    O_D(y+\perturbation)'-O_D(y)'
      = (1-O_D(y))^{-\alpha}\perturbation + \alpha(1-O_D(y))^{-(\alpha +1)}y(O_D(y+\perturbation)-O_D(y))\\
      + o(\|\perturbation\|_{\integrablefunctions^{\infty}(\timedomain;\integrablefunctions^{\infty}(\spatialdomain))})\text{ in }\integrablefunctions^{\infty}(\timedomain;\integrablefunctions^{\infty}(\spatialdomain))
  \end{multline}
  if differentiability of $y\mapsto y^{-\alpha}$ in $(0,\infty)$ for $\alpha\ge 1$ after adding a zero is considered. The first expression on the right-hand side can be recast since $w_\perturbation$ solves \eqref{eq:O_D_derivative_character} and, hence, leads to
  \begin{multline}\label{eq:O_D_derivative_character_4}
    O_D(y+\perturbation)'-O_D(y)-w_\perturbation' \\
      = \alpha(1-O_D(y))^{-(\alpha+1)}(O_D(y+\perturbation)-O_D(y)-w_\perturbation) + o(\|\perturbation\|)\text{ in }\integrablefunctions^{\infty}(\timedomain;\integrablefunctions^{\infty}(\spatialdomain)).
  \end{multline}
  Integration over $(0,\timevariable)$, standard estimates and applying Gronwal's lemma again, we can show
  \begin{equation}\label{eq:O_D_derivative_character_5}
    \|O_D(y+\perturbation)-O_D(y)-w_\perturbation\|_{\integrablefunctions^{\infty}(\timedomain;\integrablefunctions^{\infty}(\spatialdomain))}\le o(\|\perturbation\|).
  \end{equation}
  Together with \eqref{eq:O_D_derivative_character_4} this proves differentiability with $\partial O_D(y)\perturbation = w_\perturbation$.
\end{proof}
\begin{Lemma}\label{lem:auxiliary_function_is_differentiable}
  Let $f\in\continuousfunctions^{1,1}(\damageprocesses;\integrablefunctions^{\infty}(\timedomain;\integrablefunctions^{\infty}(\spatialdomain))^{\dimension^2})$. Then
  \begin{equation}\label{eq:def_auxiliary_function_1}
    \Psi_f(\damageprocess)(\timevariable,\spatialvariables):=\damageprocess(\timevariable,\spatialvariables,f(\damageprocess)(\timevariable,\spatialvariables))
  \end{equation}
  introduces a differentiable mapping
  \begin{equation}\label{eq:def_auxiliary_function_2}
    \Psi_f\colon\damageprocesses\to\integrablefunctions^{\infty}(\timedomain;\integrablefunctions^{\infty}(\spatialdomain))
  \end{equation}
  that satisfies
  \begin{equation}\label{eq:auxiliary_function_derivative}
    \partial\Psi_f(\damageprocess)h = \partial\nemytskiioperator(f(\damageprocess))\partial f(\damageprocess)h + H(f(\damageprocess)).
  \end{equation}
  for $h\in\integrablefunctions^{\infty}(\timedomain;\integrablefunctions^{\infty}(\spatialdomain;\continuousfunctions^{1,1}(\overline{Y})))$ with $\damageprocess + h\in\damageprocesses$ and $H$ denoting the respective Nemytskii operator generated by $h$, see Lemma \ref{lem:nem-op_well-posed} and Lemma \ref{lem:nem-op_differentiable}.
\end{Lemma}
\begin{proof}
  We subtract $\Psi_f(\damageprocess)$ from $\Psi_f(\damageprocess + \perturbation)$. Adding a zero, employing the chain rule, the assumptions on $f$ as well as lemmas \ref{lem:nem-op_well-posed} and \ref{lem:nem-op_differentiable} shows
  \begin{multline}\label{eq:auxiliary_function_is_differentiable}
    \Psi_f(\damageprocess + \perturbation) - \Psi_f(\damageprocess) \\
     =\partial\nemytskiioperator(f(\damageprocess))\partial f(\damageprocess)\perturbation + H(f(\damageprocess)) + \partial H(f(\damageprocess))\partial(f(\damageprocess))\perturbation + o(\|\perturbation\|_{\damageprocesses})\text{ in }\integrablefunctions^{\infty}(\timedomain;\integrablefunctions^{\infty}(\spatialdomain))^{\dimension^2}
  \end{multline}
  and completes the proof.
\end{proof}
We collect previous results and focus on showing differentiability for the parameter-to-state map.
\begin{Theorem}\label{theo:derivative_of_pts_map}
  Provided that the assumptions made in Corollary \ref{cor:Operators_decoupled} hold, the parameter-to-state map $\ptsmap$ is differentiable with
  \begin{equation}\label{eq:ptsmap_derivative}
    \begin{array}{rl}
      \partial\ptsmap\colon\damageprocesses & \to\linearoperators(\integrablefunctions^{\infty}(\timedomain;\integrablefunctions^{\infty}(\spatialdomain;\continuousfunctions^{1,1}(\overline{Y})));\integrablefunctions^{\infty}(\timedomain;\spatialfunctions)\times\damagefunctions) \\
      \damageprocess & \mapsto (\partial O_E(\ptsmap_1(\damageprocess))\identity, \identity)\circ\partial O_D(\Psi_{\mollifiedgradient\ptsmap_1(\damageprocess)}(\damageprocess))\partial\Psi_{\mollifiedgradient\ptsmap_1(\damageprocess)}(\damageprocess).
    \end{array}
  \end{equation}
  For every perturbation $\perturbation\in\integrablefunctions^{\infty}(\timedomain;\integrablefunctions^{\infty}(\spatialdomain;\continuousfunctions^{1,1}(\overline{Y})))$ of an admissible damage process $\damageprocess\in\damageprocesses$ satisfying $\damageprocess+\perturbation\in\damageprocesses$ the derivative $\partial\ptsmap(\damageprocess)\perturbation$ uniquely solves the linear system
  \begin{subequations}\label{eq:lin_prob}
    \begin{equation}\label{eq:lin_prob_EqOfMo}
      \mathcal{\elasticityoperator}(\ptsmap_2(\damageprocess))(\partial\ptsmap_1(\damageprocess)\perturbation) + \mathcal{\elasticityoperator}(1+(\partial\ptsmap_2(\damageprocess)\perturbation))(\ptsmap_1(\damageprocess))
        = 0 \text{ in } \integrablefunctions^\infty(\timedomain;\spatialfunctions^*),
    \end{equation}
    coupled with
    \begin{multline}\label{eq:lin_prob_DaEvEq}
        -(\partial\ptsmap_2(\damageprocess)\perturbation)'  + \alpha\big(1-\ptsmap_2(\damageprocess)\big)^{-(\alpha+1)} \nemytskiioperator(\mollifiedgradient\ptsmap_1(\damageprocess))(\partial\ptsmap_2(\damageprocess)\perturbation)  \\
          + \big(1-\ptsmap_2(\damageprocess)\big)^{-\alpha}\partial\nemytskiioperator\big(\mollifiedgradient\ptsmap_1(\damageprocess)\big)\mollifiedgradient(\partial\ptsmap_1(\damageprocess)\perturbation) \\
        = -\Big(1-\ptsmap_2(\damageprocess)\Big)^{-\alpha}H\Big(\mollifiedgradient\ptsmap_1(\damageprocess)\Big)\text{ in }\integrablefunctions^\infty\big(\timedomain; \integrablefunctions^\infty(\spatialdomain)\big)
    \end{multline}
    and initial condition
    \begin{equation}\label{eq:lin_prob_initial_value}
        \Big(\partial\ptsmap_2(\damageprocess)\perturbation\Big)(0) = 0\text{ in } \integrablefunctions^\infty(\spatialdomain),
    \end{equation}
  \end{subequations}
  where $\ptsmap_i:=\pi_i\circ\ptsmap$.
\end{Theorem}
\begin{proof}
  Like in previous lemmas, we have to show existence of the proposed linear operator $\partial\ptsmap(\damageprocesses)$ satisfying
  \begin{equation}\label{eq:derivative_of_pts_map_def}
    \|\ptsmap(\damageprocess + \perturbation) - \ptsmap(\damageprocess) - \partial\ptsmap(\damageprocess)\perturbation\|_{\integrablefunctions^{\infty}(\timedomain;\spatialfunctions)\times\damagefunctions} = o(\|\perturbation\|_{\damageprocesses})
  \end{equation}
  for every $\damageprocess + \perturbation\in\damageprocesses$.
  At first, we prove unique solvability of the presented characterization.
  In Lemma \ref{lem:O_E_is_differentiable} we established differentiability of $O_E$ and particularly that for every $\damage_\perturbation\in\sobolevfunctions^{1,\infty}(\timedomain;\integrablefunctions^{\infty}(\spatialdomain))$ exists a unique solution to \eqref{eq:lin_prob_EqOfMo}, i.e.,
  \begin{equation}\label{eq:derivative_of_pts_map_1}
    \partial O_E(\ptsmap_2(\damageprocess))\damage_\perturbation = O_E^{\forces(\damage_\perturbation)}(\ptsmap_2(\damageprocess))
  \end{equation}
  with $\forces(\damage_\perturbation):=-\mathcal{\elasticityoperator}(1+\damage_\perturbation)\ptsmap_1(\damageprocess)$, cf. to \eqref{eq:O_E_derivative_character_1}. Note that
  \begin{equation}\label{eq:derivative_of_pts_map_ident}
    \ptsmap_1(\damageprocess)=O_E(\ptsmap_2(\damageprocess))
  \end{equation}
  holds, that allows us to rewrite \eqref{eq:lin_prob_DaEvEq} into
  \begin{multline}\label{eq:derivative_pts_map_ode}
    -\damage_\perturbation'  + \alpha\big(1-\ptsmap_2(\damageprocess)\big)^{-(\alpha+1)} \nemytskiioperator(\mollifiedgradient\ptsmap_1(\damageprocess))\damage_\perturbation\\
    + \big(1-\ptsmap_2(\damageprocess)\big)^{-\alpha}\partial\nemytskiioperator\big(\mollifiedgradient\ptsmap_1(\damageprocess)\big)\mollifiedgradient\partial O_E(\ptsmap_2(\damageprocess))\damage_\perturbation \\
        = -\Big(1-\ptsmap_2(\damageprocess)\Big)^{-\alpha}H\Big(\mollifiedgradient\ptsmap_1(\damageprocess)\Big)\text{ in }\integrablefunctions^\infty\big(\timedomain; \integrablefunctions^\infty(\spatialdomain)\big).
  \end{multline}
  Arguing analogously to the proof of Proposition \ref{prop:DaEvEq_well-posed} and Lemma \ref{lem:O_D_is_differentiable} guarantees existence of unique solutions $\damage_\perturbation$ and, hence, $(O_E^{\forces(\damage_\perturbation)}(\ptsmap_2(\damageprocess)),\damage_\perturbation)$ to \eqref{eq:derivative_pts_map_ode} and the system in question, respectively.
  To show $O_{\spatialvariables}\in\continuousfunctions^{0,1}(\continuousfunctions(\overline{\timedomain});\continuousfunctions(\overline{\timedomain}))$ for an operator defined accordingly to Proposition \ref{prop:DaEvEq_well-posed}, we employ the results from Lemma \ref{lem:nem-op_well-posed}, \ref{lem:nem-op_differentiable}, and \ref{lem:O_E_is_differentiable}, as well as the bounds given in equations \eqref{eq:damage_functions} and \eqref{eq:damage_sources}. \\[2ex]
  The fact that $\perturbation\mapsto(\partial O_E(\ptsmap_2(\damageprocess))\damage_{\perturbation},\damage_{\perturbation})$ actually is the derivative of $\ptsmap$, satisfies the identity \eqref{eq:ptsmap_derivative}, and thus is characterized via the linear system in \eqref{eq:lin_prob} follows from applying the chain rule and lemmas \ref{lem:nem-op_differentiable}, \ref{lem:O_E_is_differentiable}, \ref{lem:O_D_is_differentiable}, \ref{lem:auxiliary_function_is_differentiable}.
\end{proof}
It is apparent that differentiability of $\ptsmap$ transfers directly to the forward operator $\forwardoperator=\embedding\circ\projection_1\circ\ptsmap$, where $\embedding\colon\integrablefunctions^{\infty}(\timedomain;\spatialfunctions)\to\integrablefunctions^2(\timedomain\times\spatialdomain)^{\dimension}$ represents the respective embedding to interpret $\integrablefunctions^{\infty}(\timedomain;\spatialfunctions)$ as a subset of $\integrablefunctions^2(\timedomain\times\spatialdomain)^{\dimension}$ (compare to \ref{def:forward_operator}).
\begin{Corollary}\label{cor:forward_operator_differentiable}
  The forward operator $\forwardoperator\colon\damageprocesses\to\integrablefunctions^2(\timedomain\times\spatialdomain)^{\dimension}$ is differentiable, i.e.,
  \begin{equation}\label{eq:forward_operator_derivative}
    \begin{array}{rl}
      \partial\forwardoperator\colon\damageprocesses & \to\linearoperators(\integrablefunctions^{\infty}(\timedomain;\integrablefunctions^{\infty}(\spatialdomain;\continuousfunctions^{1,1}(\overline{Y})));\integrablefunctions^2(\timedomain\times\spatialdomain)^{\dimension}) \\
      \damageprocess & \mapsto \partial O_E(\ptsmap_1(\damageprocess))\partial O_D(\Psi_{\mollifiedgradient\ptsmap_1(\damageprocess)}(\damageprocess))\partial\Psi_{\mollifiedgradient\ptsmap_1(\damageprocess)}(\damageprocess),
    \end{array}
  \end{equation}
  and we refer to $\partial\forwardoperator(\damageprocess)$ as the \emph{linearized forward operator in} $\damageprocess$.
\end{Corollary}
\subsection{Characterization of Adjoint}
Results on iterative regularization methods are classically obtained for situations where parameter and data set are subsets of Hilbert spaces (see \cite{KaNeSche08}) or, a bit more generalized, reflexive Banach spaces (see \cite{SchuKaHoKa12}).
Since $\damageprocesses$ is neither a Hilbert space nor reflexive one commonly considers spaces that embed compactly into reflexive Banach spaces or Hilbert spaces, respectively.
Despite the fact that it is preferable to assume as little smoothness as possible from the parameters that one wishes to identify, we will focus on $\qisobolevfunctions^s(\timedomain\times\spatialdomain\times Y)$ henceforth, since this results in a problem already difficult enough to be treated analytically (see \cite{HuSheNeuSche18} e.g.). \\[2ex]
Considering Sobolev's embedding theorem (see e.g. \cite{AdFu03}) shows that $\qisobolevfunctions^s(\timedomain\times\spatialdomain\times Y)$ embeds compactly into $\integrablefunctions^{\infty}(\timedomain\times\spatialdomain\times Y)$ for $s>N/2$.
This allows for the forward operator $\forwardoperator$ to be interpreted as an operator
\begin{equation}\label{eq:def_ptsmap_in_Hilbert_space_setting}
  \forwardoperator\colon\damageprocesses_s:=\damageprocesses\cap\qisobolevfunctions^s(\timedomain\times\spatialdomain\times Y)\to \integrablefunctions^2(\timedomain\times\spatialdomain)^{\dimension}
\end{equation}
for some $s>(N^2+N+3)/2$, since also $H^s(Y)\hookrightarrow\hookrightarrow\continuousfunctions^{1,1}(\overline{Y})$ has to hold. Because of $\damageprocesses_s\subset\damageprocesses$ all results regarding differentiability remain valid in this setting, too.
In the same spirit, we introduce
\begin{equation}\label{eq:def_damage_functions_Hilbert_space}
  \damagefunctions_s:=\damagefunctions\cap\qisobolevfunctions^s(\timedomain\times\spatialdomain)
\end{equation}
as well as elevate regularity in time in \eqref{eq:data_pts_map} for $\forces$ and $\stressforces$ to $\qisobolevfunctions^s(\timedomain;\spatialfunctions^*)$ and $\qisobolevfunctions^s(\timedomain;\boundaryfunctions^*)$, respectively.
Higher regularity in time for $\forces$, $\stressforces$ ensures that $\damage\in\damagefunctions_s$ actually holds, since $\mollifier$ only mollifies in space, thus we would only get $\nemytskiioperator(\mollifiedgradient(\ptsmap(\damageprocess)))\in\integrablefunctions^\infty(\timedomain;\qisobolevfunctions^s(\spatialdomain))$ and $\damage\in\sobolevfunctions^{1,\infty}(\timedomain;\qisobolevfunctions^s(\spatialdomain))$.%
We follow the same strategy as before when linearizing the paramter-to-state map and characterize the adjoint via differential equations they are connected to.
For any $A\in\linearoperators(X;Y)$ between two Hilbert spaces $X$,$Y$, we denote by $A^*$ the uniquely determined \emph{adjoint} via $(x,A^*)_X = (Ax,y)_Y$ for all $x\in X, y\in Y$.
We also introduce the linear operator $E_s\colon\integrablefunctions^1(\timedomain\times\spatialdomain)\to\qisobolevfunctions^s(\timedomain\times\spatialdomain)$ via
\begin{equation}\label{eq:defining_E_s}
  \langle E_s\displacements,\boldsymbol{v}\rangle_{\qisobolevfunctions^s(\timedomain\times\spatialdomain)} = \langle\displacements,\boldsymbol{v}\rangle_{\integrablefunctions^s(\timedomain\times\spatialdomain)},
\end{equation}
which is well-defined due to Sobolev's embedding theorem and Lax-Milgram's lemma.
We start characterizing the adjoint for the linearized equation of motion in Lemma \ref{lemma:adjoint_of_linEqOfMo} and look at the linearized forward operator in Theorem \ref{theo:adjoint_of_lin_forward_operator}.
\begin{Lemma}\label{lemma:adjoint_of_linEqOfMo}
  \Cref{lem:O_E_is_differentiable} introduces a well-defined linear operator
  \begin{equation}
    \partial O_E(\ptsmap_2(\damageprocess))\colon\qisobolevfunctions^s(\timedomain\times\spatialdomain) \to \integrablefunctions^2(\timedomain;\spatialfunctions).
  \end{equation}
Its adjoint can then be characterized via
\begin{equation}\label{eq:adjoint_of_linEqOfMo}
  \begin{array}{r@{\ }l}
    \partial O_E(\ptsmap_2(\damageprocess))^*\colon &\integrablefunctions^2(\timedomain;\spatialfunctions)\to \qisobolevfunctions^s(\timedomain\times\spatialdomain)\\
    \forces &\mapsto E_s\big(\elasticitytensor\strains(\ptsmap_1(\damageprocess))\colon\strains(\displacements_{\forces})\big)
  \end{array}
\end{equation}
where $\displacements_{\forces}\in\integrablefunctions^2(\timedomain;\spatialfunctions)$ solves the adjoint problem
\begin{equation}\label{eq:adjoint_eq_of_lin_EqOfMo}
  \langle\mathcal{\elasticityoperator}(\ptsmap_2(\damageprocess))\displacements_{\forces},\boldsymbol{v}\rangle_{\integrablefunctions^2(\timedomain;\spatialfunctions^{*}),\integrablefunctions^2(\timedomain;\spatialfunctions)} = \langle \forces,\boldsymbol{v}\rangle_{\integrablefunctions^2(\timedomain;\spatialfunctions)} \quad\forall\spatialtestfunctions\in\integrablefunctions^2(\timedomain;\spatialfunctions)
\end{equation}
\end{Lemma}
\begin{proof}
  Set $\boldsymbol{\spatialtestfunction} := \partial O_E(\ptsmap_2(\damageprocess))\damage$ in \eqref{eq:adjoint_eq_of_lin_EqOfMo} we immediately see that
  \begin{multline}\label{eq:adjoint_of_linEqOfMo_1}
    \langle\forces,\partial O_E(\ptsmap_2(\damageprocess))\damage\rangle_{\integrablefunctions^2(\timedomain;\spatialfunctions)}%
      = \langle\mathcal{\elasticityoperator}(\ptsmap_2(\damageprocess))\displacements_{\forces},\partial O_E(\ptsmap_2(\damageprocess))\damage\rangle_{\integrablefunctions^2(\timedomain;\spatialfunctions^*),\integrablefunctions^2(\timedomain;\spatialfunctions)} \\
      = \langle\mathcal{\elasticityoperator}(1-\damage)\ptsmap_1(\damageprocess),\displacements_{\forces}\rangle_{\integrablefunctions^2(\timedomain;\spatialfunctions^*),\integrablefunctions^2(\timedomain;\spatialfunctions)} %
      = \langle\damage,E_s(\elasticitytensor\strains(\ptsmap_1(\damageprocess)):\strains(\displacements_{\forces}))\rangle_{\qisobolevfunctions^s(\timedomain\times\spatialdomain)}
  \end{multline}
  where $E_s\colon\integrablefunctions^1(\timedomain\times\spatialdomain)\to\qisobolevfunctions^s(\timedomain\times\spatialdomain)$ is the operator from \eqref{eq:defining_E_s}.
\end{proof}
We collect the previous results to prove the main theorem on the adjoint of the linearized forward operator $\partial\forwardoperator(\damageprocess)$.
\begin{Theorem}\label{theo:adjoint_of_lin_forward_operator}
  The adjoint of the linearized forward operator $\partial\forwardoperator(\damageprocess)\colon\qisobolevfunctions^s(\timedomain\times\spatialdomain\times Y)\to\integrablefunctions^2(\timedomain\times\spatialdomain)^\dimension$ at $\damageprocess\in\damageprocesses_s$ is characterized via
  \begin{equation}\label{eq:adjoint_of_lin_forward_operator}
    \begin{array}{rl}
      \partial\forwardoperator(\damageprocess)^*\colon\integrablefunctions^2(\timedomain\times\spatialdomain)^\dimension &\to \qisobolevfunctions^s(\timedomain\times\spatialdomain\times Y) \\
                   \forces  &\mapsto \tilde\damageprocess_{\forces}
    \end{array}
  \end{equation}
  with $\tilde\damageprocess_{\forces}$ being the unique solution to
  \begin{equation}\label{eq:adjoint_of_lin_forward_operator_char}
    \Big\langle H\Big(\mollifiedgradient\ptsmap_1(\damageprocess)\Big),-\Big(1-\ptsmap_2(\damageprocess)\Big)^{-\alpha}w_e\Big\rangle_{\integrablefunctions^2(\timedomain\times\spatialdomain)}
    = \langle \perturbation, \tilde\damageprocess_{\forces}\rangle_{\qisobolevfunctions^s(\timedomain\times\spatialdomain\times Y)}
  \end{equation}
  and where $w_e\in\sobolevfunctions^{1,1}(\timedomain;\integrablefunctions^1(\spatialdomain))$ solves
  \begin{subequations}\label{eq:adjoint_linDaEvEq}
  \begin{multline}\label{eq:adjoint_linDaEvEq_1a}
    w_e'  +  \alpha\big(1-\ptsmap_2(\damageprocess)\big)^{-(\alpha + 1)} \nemytskiioperator(\mollifiedgradient\ptsmap_{\stressforces}^1(\damageprocess))w_e \\
-\partial O_E(\ptsmap_2(\damageprocess))^*(\divergence(\mollifier(\partial\nemytskiioperator\big(\mollifiedgradient\ptsmap_1(\damageprocess)\big)^{\transpose}\big(1-\ptsmap_{\stressforces}^2(\damageprocess)\big)^{-\alpha}w_e)))
      = e, \quad \text{in } \integrablefunctions^1(\timedomain\times\spatialdomain),  \end{multline}
  for $e=\elasticitytensor\strains(\ptsmap_1(\damageprocess)):\strains(\displacements_{\forces})$ with $\displacements_{\forces}$ from Lemma \ref{lemma:adjoint_of_linEqOfMo} and final value
  \begin{equation}\label{eq:adjoint_linDaEvEq_final_value}
    w_e(\timedomainmax) = 0 \quad \text{in } \integrablefunctions^1(\spatialdomain).
  \end{equation}
\end{subequations}
\end{Theorem}
\begin{proof}
We start by arguing existence and uniqueness of a solution to the adjoint problem \eqref{eq:adjoint_linDaEvEq}.
Looking at the definitions of the participating operators and previous results ensures that the coefficients of $w_e$ are at least $\integrablefunctions^\infty$ in time, thus \eqref{eq:adjoint_linDaEvEq_1a} formally makes sense.
A transformation in time via $s=\timedomainmax-\timevariable$ does not change this fact and applying the same reasoning as in the proof of Proposition \ref{prop:DaEvEq_well-posed} and Lemma \ref{lem:O_D_is_differentiable} then establishes our claim.\\[2ex]
Analog to Theorem \ref{theo:derivative_of_pts_map}, we focus on the damage evolution \eqref{eq:lin_prob_DaEvEq} in that we use the operator $\partial O_E(\forwardoperator(\damageprocess))$ to replace $\displacements_{\perturbation}$ through $\partial O_E(\forwardoperator_2(\damageprocess))\damage_{\perturbation}$, see \eqref{eq:derivative_pts_map_ode}.
Take $\displacements^{\noiselevel}\in\integrablefunctions^2(\timedomain\times\spatialdomain)^{\dimension}$.
By abuse of notation we will write $\langle u, v\rangle_{\integrablefunctions^2(\mathbb{R}^{\dimension})} = \int_{\mathbb{R}^{\dimension}} u v\,d\spatialvariables$ for integrablefunctions $uv$ even if, without loss of generality, $v$ does not belong to $\integrablefunctions^2(\mathbb{R}^{\dimension})$.
Because of
\begin{multline}\label{eq:adjoint_linDaEvEq_1}
  \langle\partial\forwardoperator(\damageprocess)\perturbation,\displacements^{\noiselevel}\rangle_{\integrablefunctions^2(\timedomain\times\spatialdomain)^{\dimension}}%
  = \langle\partial\ptsmap_1(\damageprocess)\perturbation,\displacements^{\noiselevel}\rangle_{\integrablefunctions^2(\timedomain\times\spatialdomain)^{\dimension}} %
  = \langle\partial O_E(\ptsmap_2(\damageprocess))\partial\ptsmap_2(\damageprocess)\perturbation,\displacements^{\noiselevel}\rangle_{\integrablefunctions^2(\timedomain\times\spatialdomain)^{\dimension}}\\
  = \langle\partial\ptsmap_2(\damageprocess)\perturbation,\partial O_E(\ptsmap_2(\damageprocess))^*\displacements_{\noiselevel}\rangle_{\qisobolevfunctions^s(\timedomain\times\spatialdomain)} %
  = \langle\partial\ptsmap_2(\damageprocess)\perturbation,\elasticitytensor\strains(\ptsmap_1(\damageprocess)):\strains(\boldsymbol{v}_{\displacements^{\noiselevel}})\rangle_{\integrablefunctions^2(\timedomain\times\spatialdomain)},
\end{multline}
where we took $\integrablefunctions^2(\timedomain;\spatialfunctions)$ as a subset of $\integrablefunctions^2(\timedomain\times\spatialdomain)^{\dimension}$, it is sufficient to characterize the adjoint of $\partial\ptsmap_2(\damageprocess)\colon\qisobolevfunctions^s(\timedomain\times\spatialdomain\times Y)\to\integrablefunctions^2(\timedomain\times\spatialdomain)$.
The first two equalities in \eqref{eq:adjoint_linDaEvEq_1} follow from definitions of $\forwardoperator$ and $\partial O_E(\ptsmap_2(\damageprocess))$ in Definition \ref{def:forward_operator} and Corollary \ref{cor:Operators_decoupled}, respectively, as well as Theorem \ref{theo:derivative_of_pts_map}.
By replacing $\integrablefunctions^2(\timedomain;\spatialfunctions)$ with $\integrablefunctions^2(\timedomain\times\spatialdomain)^{\dimension}$ and $\displacements_{\forces}$ through $\boldsymbol{v}_{\displacements^{\noiselevel}}$ in \eqref{eq:adjoint_eq_of_lin_EqOfMo}, we can argue validity of the last equality. \\[2ex]
Let us take $\damage_{\perturbation}:=\partial\ptsmap_2(\damageprocess)\perturbation$ and multiply \eqref{eq:derivative_pts_map_ode} with $w\in\sobolevfunctions^{1,1}(\timedomain;\integrablefunctions^1(\spatialdomain))$. We then integrate over $\timedomain\times\spatialdomain$ shows
\begin{multline}\label{eq:motivating_adjoint_ode}
  \int_0^{\timedomainmax}\Big\langle-\damage_\perturbation',w\Big\rangle_{\integrablefunctions^2(\spatialdomain)} %
  + \Big\langle\alpha\big(1-\ptsmap_2(\damageprocess)\big)^{-(\alpha+1)} \nemytskiioperator(\mollifiedgradient\ptsmap_1(\damageprocess))\damage_\perturbation,w\Big\rangle_{\integrablefunctions^2(\spatialdomain)}\\
    + \Big\langle\big(1-\ptsmap_2(\damageprocess)\big)^{-\alpha}\partial\nemytskiioperator\big(\mollifiedgradient\ptsmap_1(\damageprocess)\big)\mollifiedgradient\partial O_E(\ptsmap_2(\damageprocess))\damage_\perturbation,w\Big\rangle_{\integrablefunctions^2(\spatialdomain)}\,d\timevariable \\
        = -\Big\langle\Big(1-\ptsmap_2(\damageprocess)\Big)^{-\alpha}H\Big(\mollifiedgradient\ptsmap_1(\damageprocess)\Big),w\Big\rangle_{\integrablefunctions^2(\timedomain\times\spatialdomain)}.
\end{multline}
We use integration by parts to establish
\begin{equation}\label{eq:motivating_adjoint_ode_1st_term}
  \int_0^{\timedomainmax}\langle -\damage_{\perturbation}',w\rangle_{\integrablefunctions^2(\spatialdomain)}\,d\timevariable %
    = \int_0^{\timedomainmax}\langle \damage_{\perturbation}',w\rangle_{\integrablefunctions^2(\spatialdomain)}\,d\timevariable + \langle\damage_{\perturbation}(0),w(0)\rangle_{\integrablefunctions^2(\spatialdomain)} - \langle\damage_{\perturbation}(\timedomainmax),w(\timedomainmax)\rangle_{\integrablefunctions^2(\spatialdomain)}
\end{equation}
for the first term on the left-hand side. The second term obviously satisfies
\begin{equation}\label{eq:motivating_adjoint_ode_2nd_term}
  \Big\langle\alpha\big(1-\ptsmap_2(\damageprocess)\big)^{-(\alpha+1)} \nemytskiioperator(\mollifiedgradient\ptsmap_1(\damageprocess))\damage_\perturbation,w\Big\rangle_{\integrablefunctions^2(\timedomain\times\spatialdomain)}
  = \Big\langle\damage_\perturbation,\alpha\big(1-\ptsmap_2(\damageprocess)\big)^{-(\alpha+1)} \nemytskiioperator(\mollifiedgradient\ptsmap_1(\damageprocess))w\Big\rangle_{\integrablefunctions^2(\timedomain\times\spatialdomain)},
\end{equation}
and the third one we can recast into
\begin{multline}\label{eq:motivating_adjoint_ode_3rd_term}
  \Big\langle\big(1-\ptsmap_2(\damageprocess)\big)^{-\alpha}\partial\nemytskiioperator\big(\mollifiedgradient\ptsmap_1(\damageprocess)\big)\mollifiedgradient\partial O_E(\ptsmap_2(\damageprocess))\damage_\perturbation,w\Big\rangle_{\integrablefunctions^2(\timedomain\times\spatialdomain)} \\
    = \int_{\timedomain\times\spatialdomain}\mollifiedgradient\partial O_E(\ptsmap_2(\damageprocess))\damage_\perturbation:\partial_{\boldsymbol{y}}\damageprocess(\cdot,\mollifiedgradient\ptsmap_1(\damageprocess)(\cdot))\big(1-\ptsmap_2(\damageprocess)\big)^{-\alpha}w\,d(\timevariable,\spatialvariables) \qquad\qquad\\
    = -\langle\partial O_E(\ptsmap_2(\damageprocess))\damage_\perturbation,\divergence(\mollifier(\partial\nemytskiioperator(\mollifiedgradient\ptsmap_1(\damageprocess))^{\transpose}\big(1-\ptsmap_2(\damageprocess)\big)^{-\alpha}w))\rangle_{\integrablefunctions^2(\timedomain\times\spatialdomain)^{\dimension}}\\
    = \langle\damage_\perturbation,-\partial O_E(\ptsmap_2(\damageprocess))^*\divergence(\mollifier(\partial\nemytskiioperator(\mollifiedgradient\ptsmap_1(\damageprocess))^{\transpose}\big(1-\ptsmap_2(\damageprocess)\big)^{-\alpha}w))\rangle_{\integrablefunctions^2(\timedomain\times\spatialdomain)^{\dimension}}
\end{multline}
taking $\partial O_E(\ptsmap_2(\damageprocess)^*$ from Lemma \ref{lemma:adjoint_of_linEqOfMo}.
Let $w_e\in\sobolevfunctions^{1,1}(\timedomain;\integrablefunctions^1(\spatialdomain))$ be the unique solution to \eqref{eq:adjoint_linDaEvEq} with $e:=\elasticitytensor\strains(\ptsmap_1(\damageprocess)):\strains(\boldsymbol{v}_{\displacements^{\noiselevel}})\in\integrablefunctions^1(\timedomain\times\spatialdomain)$.
In considering \eqref{eq:lin_prob_initial_value} and \eqref{eq:adjoint_linDaEvEq_final_value}, too, we can argue the vanishing of $\langle\damage_{\perturbation}(0),w(0)\rangle_{\integrablefunctions^2(\spatialdomain)}$ and $\langle\damage_{\perturbation}(\timedomainmax),w(\timedomainmax)\rangle_{\integrablefunctions^2(\spatialdomain)}$ in \eqref{eq:motivating_adjoint_ode_1st_term} and use this to find
\begin{multline}\label{eq:characterize_adjoint}
  \langle\partial\forwardoperator(\damageprocess)\perturbation,\displacements^{\noiselevel}\rangle_{\integrablefunctions^2(\timedomain\times\spatialdomain)^{\dimension}} %
    = \langle\damage_{\perturbation},e\rangle_{\integrablefunctions^2(\timedomain\times\spatialdomain)} %
    = -\Big\langle H\Big(\mollifiedgradient\ptsmap_1(\damageprocess)\Big),\Big(1-\ptsmap_2(\damageprocess)\Big)^{-\alpha}w_e\Big\rangle_{\integrablefunctions^2(\timedomain\times\spatialdomain)}\\
    = \Big\langle H\Big(\mollifiedgradient\ptsmap_1(\damageprocess)\Big),-\Big(1-\ptsmap_2(\damageprocess)\Big)^{-\alpha}w_e\Big\rangle_{\integrablefunctions^2(\timedomain\times\spatialdomain)}
    = \langle \perturbation, \partial\forwardoperator(\damageprocess)^*\displacements^{\noiselevel}\rangle_{\qisobolevfunctions^s(\timedomain\times\spatialdomain\times Y)},
\end{multline}
where we collected the results from \eqref{eq:motivating_adjoint_ode} to \eqref{eq:motivating_adjoint_ode_3rd_term} as well as well-posedness from \eqref{eq:adjoint_linDaEvEq}, Theorem \ref{theo:derivative_of_pts_map}, and Lemma \ref{lemma:adjoint_of_linEqOfMo}.
It is easy to verify that the right-hand side is a linear functional on $\qisobolevfunctions^s(\timedomain\times\spatialdomain\times Y)$ so that we can employ Lax-Milgram's lemma once again in the last step to be able to characterize the adjoint which completes this proof.
\end{proof}

%% file: inverse_problem.tex
In this section we address some questions regarding the inverse problem, which we will state in Definition \ref{def:inverse_problem}.
Namely this means ill-posedness to justify the need for regularization methods and, since we are faced with a nonlinear problem, we show a strong nonlinearity condition to ensure convergence for iterative regularization methods.
\begin{Definition}[Inverse Problem]\label{def:inverse_problem}
  Let $\displacements^\noiselevel\in\integrablefunctions^2(\timedomain\times\spatialdomain)^{\dimension}$ denote measurements of displacements satisfying
  \begin{equation}\label{eq:defining_noise_level}
    \|\displacements-\displacements^\noiselevel\|_{\integrablefunctions^2(\timedomain\times\spatialdomain)^{\dimension}} \le \noiselevel,
  \end{equation}
  where $\noiselevel\ge 0$ represents the noise level. With $\forwardoperator$ denoting the operator introduced in \Cref{def:forward_operator}, we search for a function $\damageprocess$ in $\damageprocesses_s=\domain(\forwardoperator)$ according to \eqref{eq:damage_sources} and \eqref{eq:def_ptsmap_in_Hilbert_space_setting} satisfying
\begin{equation}\label{eq:inverse_problem}
  \forwardoperator(\damageprocess) = \displacements^{\noiselevel}\text{ in }\integrablefunctions^2(\timedomain\times\spatialdomain)^{\dimension}.
\end{equation}
  Its linearization then is to find $\perturbation\in\damageprocesses_s$ such that
  \begin{equation}\label{eq:lin_inverse_problem}
    \partial\forwardoperator(\damageprocess)\perturbation = \displacements^{\noiselevel}\text{ in }\integrablefunctions^2(\timedomain\times\spatialdomain)^{\dimension}
  \end{equation}
  holds.
\end{Definition}
\subsection{Ill-posedness}\label{sec:Ill_Posedness_of_Nonlinear_Inverse_Problem}
Arguing ill-posedness of inverse problems is often done successfully by showing compactness of one of the otherwise continuous participating operators the forward operator can be dissected into and by employing embedding theorems.
In a linear setting like \eqref{eq:lin_inverse_problem} operators often exhibit a regularizing effect on the parameter as the solution tends to lie in a space that embeds compactly into the space connected to measurements (see e.g. \cite{JiMa12}).
A linear setting with compact operator then is (globally) ill-posed if and only if its range is of infinite dimension. In our setting, The solution does not embed compactly into the measurements.
We would need more smoothness in time to get a compact embedding $\mathcal{I}(\forwardoperator'(\damageprocess^{\dagger}))\subset\integrablefunctions^{\infty}(\timedomain;\spatialfunctions)$ into the data $\integrablefunctions^2(\timedomain\times\spatialdomain)^{\dimension}$, see \cite{Je-PiAu63} or, extending the previous result, \cite{Si86}.
Also note Remark \ref{rem:time_regularity} item \eqref{rem:time_regularity:on_right_hand_side}.
In the linear as well as the non-linear setting, we immediately see that the operators $\forwardoperator$ and $\partial\forwardoperator(\damageprocess)$, respectively, are compact, i.e., bounded sets are mapped into relatively compact subsets, which stems directly from the embedding $\damageprocesses_s\hookrightarrow\hookrightarrow\damageprocesses$ for suitable $s$.
This establishes the following result.
\begin{Remark}\label{rem:time_regularity}
  \begin{inparaenum}[(a)]
    \item\label{rem:time_regularity:on_right_hand_side}%
    As time is merely a parameter in the balances of forces and damage functions are already smooth enough, one could remedy the situation by assuming a touch more smoothness in time on the right-hand side in the balance equation, which would not be that much of a drawback. \\[2ex]
    \item\label{rem:time_regularity:connection_to_linearization}%
    Different concepts of ill-posedness are discussed in \cite{BeHo00} together with their inter-dependence as well as the connection between a nonlinear problem and its linearization.
  \end{inparaenum}
\end{Remark}
There are different concepts of ill-posedness found in the literature, see \cite{BeHo00, HoSche98}, for example, and note Remark \ref{rem:time_regularity} item \eqref{rem:time_regularity:connection_to_linearization}.
Since we primarily face a nonlinear inverse problem and ill-posedness becomes a local property then, we follow \cite{BeHo00} in his presentation.
\begin{Proposition}\label{prop:inverse_problem_ill-posed}
  The inverse problem from Defintition \ref{def:inverse_problem} is locally ill-posed in $\damageprocess\in\damageprocesses_s$, i.e., for any arbitrarily small $\radius > 0$ exists a sequence $(\damageprocess_n)_{n\in\mathbb{N}}\in\ball_{\damageprocesses_s}(\damageprocess;\radius)\subset\damageprocesses_s$ such that
  \begin{equation}\label{eq:locally_ill-posed}
    \forwardoperator(\damageprocess_n)\to\forwardoperator(\damageprocess)\text{ in }\integrablefunctions^2(\timedomain\times\spatialdomain)^{\dimension},\quad\damageprocess_n \not\to \damageprocess\text{ in }\damageprocesses_s\qquad (n\to\infty).
  \end{equation}
  holds.
\end{Proposition}
\begin{proof}
  The fact that $\qisobolevfunctions^s(\timedomain\times\spatialdomain\times Y)$ is a Hilbert space of infinite dimension yields the existence of a complete orthonormal system $(\varphi_n)_{n\in\mathbb{N}}$ such that $\perturbation\in\qisobolevfunctions^s(\timedomain\times\spatialdomain\times Y)$ may be decomposed to $\perturbation=\sum_{n=1}^{\infty}\langle\perturbation,\varphi_n\rangle\varphi_n$ with $\|\perturbation\|_{\damageprocesses_s}^2 = \sum_{n=1}^{\infty}\langle\perturbation,\varphi_n\rangle^2 < \infty$ as well as $\langle\perturbation,\varphi_n\rangle\to 0$ as $n\to\infty$.
  For any $0<\rho\le\radius$ we have $\damageprocess_n:=\damageprocess + \rho\varphi_n\in\ball_{\damageprocesses_s}(\damageprocess;\rho)$ entailing $\varphi_n\rightharpoonup 0$, $\damageprocess_n\rightharpoonup\damageprocess$, as well as $\|\damageprocess_n - \damageprocess\| = \rho > 0$ for $n\to\infty$.
  Due to Sobolev's embedding theorem, we get strong convergence for $(\damageprocess_n)_{n\in\mathbb{N}}$ in $\qisobolevfunctions^{s-1}(\timedomain\times\spatialdomain\times Y)\subset\integrablefunctions^\infty(\timedomain;\integrablefunctions^\infty(\spatialdomain;\continuousfunctions^{0,1}(\overline{Y})))$.
  In combination with \eqref{eq:estimate_displacments_via_damageprocesses} this yields \eqref{eq:locally_ill-posed}.
\end{proof}
The interdependence between local ill-posedness of nonlinear inverse problems and their linearization was thoroughly investigated in \cite{HoSche94, HoSche98}.
The following result is htaken from these references then yields ill-posedness of the linearization.
\begin{Proposition}\label{prop:lin_inverse_problem_ill-posed}
  Let $X$, $Y$ denote Hilbert spaces and let $\forwardoperator\colon\domain(\forwardoperator)\subset X\to Y$ be differentiable with locally Lipschitz derivative in $\damageprocess\in\interior(\domain(\forwardoperator))$.
  If \eqref{eq:inverse_problem} is locally ill-posed, so is its linearization \eqref{eq:lin_inverse_problem}, too.
\end{Proposition}
\begin{Corollary}\label{cor:lin_inverse_problem_ill-posed}
  The linearized inverse problem \eqref{eq:lin_inverse_problem} is (globally) ill-posed.
\end{Corollary}
\subsection{Strong Nonlinearity Condition/Tangential Cone Condition}\label{sec:strong_tangential_cone_condition}
The \emph{tangential cone condition}, i.e.,
\begin{multline}\label{eq:tangential_cone_condition}
  \|\forwardoperator(\damageprocess') - \forwardoperator(\damageprocess) - \partial\forwardoperator(\damageprocess)(\damageprocess' - \damageprocess)\|_{\integrablefunctions^2(\timedomain\times\spatialdomain)^{\dimension}} \\
  \le\eta\|\forwardoperator(\damageprocess)-\forwardoperator(\damageprocess')\|_{\integrablefunctions^2(\timedomain\times\spatialdomain)^{\dimension}},\ \eta<\frac{1}{2}\text{ for all }\damageprocess,\ \damageprocess'\in\ball(\damageprocess_0;\radius)\subset\domain(\forwardoperator),
\end{multline}
is a local property in a ball $\ball(\damageprocess;\radius)$ for some $\damageprocess_0\in\damageprocesses_s$ and essential to ensure convergence for iterative methods like the nonlinear Landweber iteration, see for example \cite{HaNeuSche95}, where the authors investigate the convergence behavior of said method.
To get optimal convergence results and also in case of inexact Newton iterations more restrictive assumptions are necessary, see \cite{HaNeuSche95, AnRi99, AnRi01}.
We start by showing the main theorem from which we then can infer \eqref{eq:tangential_cone_condition}, the tangential cone condition. We adept the strategy used in \cite{HuSheNeuSche18} to our setting.
\begin{Remark}
  To establish an estimate like \eqref{eq:pre_tangential_cone_condition}, we need higher regularity for the displacements, at least $\qisobolevfunctions^2(\spatialdomain)^{\dimension}$ in space and just Dirichlet- or Neumann type boundary conditions for the direct problem.
  In general, mixed boundary conditions are not possible since we can only expect $\displacements\in\continuousfunctions^{\frac{2}{3}-\varepsilon}(\spatialdomain)$ cf.~\cite{GiSa97}.
  Robin type boundary conditions will not work in any dimension (see \eqref{eq:pre_tangential_cone_condition_2}).
\end{Remark}
\begin{Theorem}\label{theo:pre_tangential_cone_condition}
There exists a constant $c>0$ such thtat for each $\damageprocess,\perturbation\in\damageprocesses_s$
\begin{equation}\label{eq:pre_tangential_cone_condition}
  \|\forwardoperator(\damageprocess + \perturbation) - \forwardoperator(\damageprocess) - \partial\forwardoperator(\damageprocess)\perturbation\|_{\integrablefunctions^2(\timedomain\times\spatialdomain)^{\dimension}} %
  \le c\|\perturbation\|_{\damageprocesses}\|\forwardoperator(\damageprocess + \perturbation)-\forwardoperator(\damageprocess)\|_{\integrablefunctions^2(\timedomain\times\spatialdomain)^{\dimension}}
\end{equation}
holds.
\end{Theorem}
\begin{proof}
  We take $w\in\integrablefunctions^2(\timedomain\times\spatialdomain)^{\dimension}$ and $\damageprocess,\perturbation\in\qisobolevfunctions^s(\timedomain\times\spatialdomain\times Y)$ such that $\damageprocess$, $\damageprocess + \perturbation \in\damageprocesses_s$. Realizing that $\forwardoperator(\damageprocess + \perturbation) - \forwardoperator(\damageprocess) - \partial\forwardoperator(\damageprocess)\perturbation \in\spatialfunctions\subset\integrablefunctions^2(\timedomain\times\spatialdomain)^{\dimension}$ is an admissible test function, we get
  \begin{multline}\label{eq:pre_tangential_cone_condition_1}
    \langle \forwardoperator(\damageprocess + \perturbation) - \forwardoperator(\damageprocess) - \partial\forwardoperator(\damageprocess)\perturbation, w \rangle_{\integrablefunctions^2(\timedomain\times\spatialdomain)^{\dimension}} \\
    = \langle \timedependentelasticityoperator(\ptsmap_2(\damageprocess))\displacements_w, \ptsmap_1(\damageprocess + \perturbation) - \ptsmap_1(\damageprocess) - \partial\ptsmap_1(\damageprocess)\perturbation \rangle \\
    =\langle \timedependentelasticityoperator(1+(\ptsmap_2(\damageprocess + \perturbation) + \ptsmap_2(\damageprocess)))(\ptsmap_1(\damageprocess + \perturbation) - \ptsmap_1(\damageprocess) ,\displacements_w \rangle,
  \end{multline}
  where $\displacements_w=O_E^w(\ptsmap_2(\damageprocess))$ denotes the unique solution guaranteed by Proposition \ref{prop:EqOfMo_well-posed}.
  Here, employing symmetric properties of $\timedependentelasticityoperator\colon\spatialfunctions\to\spatialfunctions^*$ and a suitable version of the identity presented in \eqref{eq:O_E_derivative_eq_1} then established the latter equality.
  Turning to its definition, we can then rewrite this equation and use integration by parts to obtain the following
  \begin{multline}\label{eq:pre_tangential_cone_condition_2}
    \langle \timedependentelasticityoperator(1+(\ptsmap_2(\damageprocess + \perturbation) + \ptsmap_2(\damageprocess)))(\ptsmap_1(\damageprocess + \perturbation) - \ptsmap_1(\damageprocess) ,\displacements_w \rangle \\
      = \int_0^{\timedomainmax}\int_{\spatialdomain} (\ptsmap_2(\damageprocess +\perturbation) - \ptsmap_2(\damageprocess))\elasticitytensor\strains(\ptsmap_1(\damageprocess+\perturbation)-\ptsmap_1(\damageprocess)):\strains(\displacements_w)\,d\spatialvariables\,d\timevariable \qquad\qquad\\
      = -\int_0^{\timedomainmax}\int_{\spatialdomain} \divergence((\ptsmap_2(\damageprocess +\perturbation) - \ptsmap_2(\damageprocess))\elasticitytensor\strains(\displacements_w))(\ptsmap_1(\damageprocess+\perturbation)-\ptsmap_1(\damageprocess))\,d\spatialvariables\,d\timevariable \\
     + \int_0^{\timedomainmax}\int_{\partial\spatialdomain} (\ptsmap_2(\damageprocess +\perturbation) - \ptsmap_2(\damageprocess))\elasticitytensor\strains(\displacements_w)\normals
      \cdot(\ptsmap_1(\damageprocess + \perturbation) - \ptsmap_1(\damageprocess)) \,d\mathcal{H}^{\dimension-1}\,d\timevariable.
  \end{multline}
  Since $\displacements_w\in\qisobolevfunctions^2(\timedomain\times\spatialdomain)^{\dimension}$ entails $(1-\ptsmap_2(\damageprocess))\elasticitytensor\strains(\displacements_w)\normals = 0$ a.e.~on $\partial\spatialdomain$, together with \eqref{eq:damage_functions} we can argue $\elasticitytensor\strains(\displacements_w)\normals = 0$ and therefore see that the boundary integral vanishes in \eqref{eq:pre_tangential_cone_condition_2}.
  Please note that this argument does not hold for Robin-type boundary conditions.
  This leaves us with
  \begin{multline}\label{eq:pre_tangential_cone_condition_3}
    \langle \forwardoperator(\damageprocess + \perturbation) - \forwardoperator(\damageprocess) - \partial\forwardoperator(\damageprocess)\perturbation, w \rangle_{\integrablefunctions^2(\timedomain\times\spatialdomain)^{\dimension}} \\
    = -\int_0^{\timedomainmax}\int_{\spatialdomain} \divergence((\ptsmap_2(\damageprocess +\perturbation) - \ptsmap_2(\damageprocess))\elasticitytensor\strains(\displacements_w))(\ptsmap_1(\damageprocess+\perturbation)-\ptsmap_1(\damageprocess))\,d\spatialvariables\,d\timevariable \hspace{2.7cm}\\
    \le \int_0^{\timedomainmax} \| \divergence((\ptsmap_2(\damageprocess +\perturbation) - \ptsmap_2(\damageprocess))\elasticitytensor\strains(\displacements_w))\|_{\integrablefunctions^2(\spatialdomain)^{\dimension}}\|\ptsmap_1(\damageprocess+\perturbation)-\ptsmap_1(\damageprocess))\|_{\integrablefunctions^2(\spatialdomain)^{\dimension}}\,d\timevariable \qquad\\
    \le C \|\ptsmap_2(\damageprocess + \perturbation)-\ptsmap_2(\damageprocess)\|_{\integrablefunctions^\infty(\timedomain;\sobolevfunctions^{1,\infty}(\spatialdomain))} \| w \|_{\integrablefunctions^2(\timedomain\times\spatialdomain)^{\dimension}} \|\forwardoperator(\damageprocess + \perturbation) - \forwardoperator(\damageprocess)\|_{\integrablefunctions^2(\timedomain\times\spatialdomain)^{\dimension}}.
  \end{multline}
  Note that we used $\|\displacements_w\|_{\qisobolevfunctions^2(\spatialdomain)^{\dimension}}\le\constant(1-\damageconstant_0)^{-1}\|w\|_{\integrablefunctions^2(\spatialdomain)^{\dimension}}$ where $c$,$\damageconstant_0$ are indepentent of $\damageprocess$ and $\perturbation$.
  Remember that $\damageconstant_0$ denotes the constant from \eqref{eq:damage_functions}.
  Using results on higher regularity and
  \begin{equation}
    \|\forwardoperator(\damageprocess + \perturbation) - \forwardoperator(\damageprocess) - \partial\forwardoperator(\damageprocess)\perturbation\|_{\integrablefunctions^2(\timedomain\times\spatialdomain)^{\dimension}}
    = \sup_{\scriptscriptstyle{\substack{\|w\|=1, \\ \int_{\spatialdomain}w\,d\spatialvariables=0}}} \langle \forwardoperator(\damageprocess + \perturbation) - \forwardoperator(\damageprocess) - \partial\forwardoperator(\damageprocess)\perturbation, w \rangle_{\integrablefunctions^2(\timedomain\times\spatialdomain)^{\dimension}}
  \end{equation}
  reveals
  \begin{equation}\label{eq:full_tangential_cone_condition}
    \|\forwardoperator(\damageprocess + \perturbation) - \forwardoperator(\damageprocess) - \partial\forwardoperator(\damageprocess)\perturbation\|_{\integrablefunctions^2(\timedomain\times\spatialdomain)^{\dimension}}\le \constant \|h\|_{\damageprocesses_s} \|\forwardoperator(\damageprocess + \perturbation) - \forwardoperator(\damageprocess)\|_{\integrablefunctions^2(\timedomain\times\spatialdomain)^{\dimension}}
  \end{equation}
  and completes the proof.
\end{proof}
\begin{Remark}
  Even if the constants dependet on $\damageprocess$ the result would be strong enough to ensure convergence of the Landweber iteration under the additional assumption that the initial guess $\damageprocess_0$ is chosen close enough to $\damageprocess^{\dagger}$.
  The proof can be found in \cite{HaNeuSche95,KaNeSche08}.\\[2ex]
\end{Remark}
\begin{Corollary}
  The operator $\forwardoperator$ introduced in \Cref{def:forward_operator} satisifes the \emph{tangential cone condition} from \eqref{eq:tangential_cone_condition}.
\end{Corollary}
\begin{proof}
  Taking the constant $c$ from Theorem \ref{theo:pre_tangential_cone_condition} and any $\damageprocess_0\in\damageprocesses_s$ we choose $r:=c/4$, $\damageprocess,\damageprocess'\in\ball(\damageprocess_0;r)$ and $h:=\damageprocess-\damageprocess'$.
\end{proof}

%% file: outlook.tex
After having motivated the forward problem in the shape of a fully-coupled, strongly nonlinear system of differential equations describing damage evolution in a quasi-static elastic setting, we thoroughly investigated the forward operator. 
We showed well-posedness, regularity, differentiability, and, in addition, characterized the Hilbert adjoint of the linearized problem.
We provided all the tools to successfully present the inverse problem and to analyze it in depth. 
We were able to show that it is ill-posed and by verifying the hard to come by tangential cone condition, we proved a sufficient condition for the nonlinear Landweber method and a necessary condition for many iterative Hilbert space methods.\\[2ex]
While we will investigate the numerical implementation in the second part, we want to give some ideas for possible future research.
The right-hand side of the damage evolution equation (see. \eqref{eq:DaEvEq}) still imposes some structure on the damage processes $\damageprocess\in\damageprocesses$.
In a next step it would be rather straight forward to incorporate the damage variable into the process, i.e., 
\begin{equation}
    \damage' = \nemytskiioperator(\damage,\mollifiedgradient\displacements)\text{ in }\integrablefunctions^{\infty}(\timedomain;\integrablefunctions^{\infty}(\spatialdomain)).
\end{equation}
As mentioned before, it might be beneficial to change the damage evolution equation setting into an $\integrablefunctions^p$-setting.
This would help getting rid of the need to regularize the displacements' gradient.
In regard of the inverse problem, we only focused on the solution operator of the forward problem. 
In many applications in elasticity, measurements are taken on part of the boundary of the domain. 
But getting comparable results with measurements on part of the boundary is not easy to come by.